\documentclass[]{interact}

\usepackage{epstopdf}
\usepackage[caption=true]{subfig}

\usepackage[numbers,sort&compress]{natbib}
\bibpunct[, ]{[}{]}{,}{n}{,}{,}
\renewcommand\bibfont{\fontsize{10}{12}\selectfont}

\usepackage{mathrsfs,amssymb,algorithm,algorithmic,url,longtable,color}

\theoremstyle{plain}
\newtheorem{theorem}{Theorem}[section]
\newtheorem{lemma}[theorem]{Lemma}
\newtheorem{corollary}[theorem]{Corollary}

\newtheorem{assumption}[theorem]{Assumption}

\theoremstyle{definition}
\newtheorem{definition}[theorem]{Definition}

\theoremstyle{remark}

\newcommand{\Lcal}{\mathcal{L}} 
\newcommand{\Ecal}{\mathcal{E}} 
\newcommand{\Ical}{\mathcal{I}} 
\newcommand{\Acal}{\mathcal{A}} 
\newcommand{\Ncal}{\mathcal{N}} 
\newcommand{\baligned}{\begin{aligned}} 
	\newcommand{\ealigned}{\end{aligned}} 
\newcommand{\bassumption}{\begin{assumption}} 
	\newcommand{\eassumption}{\end{assumption}} 
\newcommand{\st}{\text{s.t.} } 

\newcommand{\zbf}{\boldsymbol{z}} 
\newcommand{\ubf}{\boldsymbol{u}} 
\newcommand{\vbf}{\boldsymbol{v}} 

\newcommand{\support}[2]{\delta^*(#1\mid #2)}
\begin{document}

	
	\articletype{ARTICLE TEMPLATE}
	
	\title{An Inexact First-Order Method for   Constrained Nonlinear Optimization}
	
\author{
  \name{Hao Wang\textsuperscript{a}\thanks{CONTACT: {Hao Wang}. Email: haw309@gmail.com} and   Fan Zhang\textsuperscript{b}  and Jiashan Wang\textsuperscript{c} and Yuyang Rong\textsuperscript{d}}
  \affil{\textsuperscript{a, b, d}School  of Information Science and Technology, ShanghaiTech University, Shanghai, China; \\ 
   \textsuperscript{b}Shanghai Institute of Microsystem and Information Technology, Chinese
   Academy of Sciences, Shanghai, China;\\
   \textsuperscript{b}University of Chinese Academy of Sciences, Beijing, China;\\
   \textsuperscript{c}Department of Mathematics, University of Washington.}
 }
	\maketitle
	
	\begin{abstract}
		The primary focus of this paper is on designing an inexact first-order algorithm for solving   constrained nonlinear optimization problems. By controlling the inexactness of the  subproblem solution, we can significantly reduce the computational cost needed for each iteration. A penalty parameter updating strategy {\color{blue}during the process of solving the subproblem} enables the algorithm to automatically detect infeasibility. Global convergence for both feasible and infeasible cases are proved. Complexity analysis for the KKT residual is also derived under {\color{blue}mild} assumptions. Numerical experiments exhibit the ability of the proposed algorithm to rapidly find inexact optimal solution through cheap computational cost.
	\end{abstract}
	
	\begin{keywords}
		nonlinear optimization, sequential linear optimization,   constrained problems,  exact penalty functions, convex composite optimization, first-order methods
	\end{keywords}
	
	\section{Introduction} \label{s.Pd}
	In the last few years, a number of advances on first-order optimization methods have been made for unconstrained/constrained optimization problems in a wide range of applications including machine learning, compressed sensing and signal processing. This is largely due to their relatively low iteration computational cost, as well as their implementation easiness. 
	Numerous works have emerged for solving unconstrained optimization problems, e.g. the stochastic gradient descent methods \cite{ bottou2004stochastic, bottou2010large,bottou2012stochastic, johnson2013accelerating} and mirror descent methods \cite{beck2003mirror, juditsky2011solving} {\color{blue}for} solving machine learning problems, soft-thresholding type algorithms \cite{beck2009fast, daubechies2004iterative} for solving sparse reconstruction problems. 
	For certain structured constrained optimization problems, many first-order methods have also captured researchers' attention, such as  
	conditional gradient  methods (also known as Frank-Wolfe methods) for solving principle component analysis problems \cite{luss2013conditional}, gradient projection methods for solving various 
	problems with structured constraints \cite{duchi2011adaptive,xiao2014proximal}, and gradient methods on Riemannian manifolds \cite{ absil2009optimization, udriste1994convex}.

	On the contrary, little attention has been paid on first-order methods for solving general constrained  optimization problems in the past decades. 
	This is mainly due to  the \emph{slow tail convergence} of first-order methods, since it can cause heavy  computational burden for obtaining accurate solutions.  Most of the research efforts  can date back to the successive linear programming (SLP) algorithms \cite{baker1985successive, lasdon1979solving} designed in 1960s-1980s for solving pooling problems in oil refinery. Among various SLP algorithms, the most famous SLP algorithm is proposed by Fletcher and Maza in   \cite{fletcher2013practical, Fletcher}, which analyzes the global convergence as well as the local convergence under strict complementarity, second-order sufficiency and regularity conditions.  The other well-known work is the active-set algorithmic option implemented in the off-the-shelf solver  \texttt{Knitro} \cite{byrd2003algorithm}, which sequentially solves a linear optimization subproblem and an equality constrained quadratic optimization subproblem.

	The primary focus of this paper is to design an algorithmic framework of first-order methods for   nonlinear constrained optimization. Despite of their weakness on tail convergence,  first-order methods are widely used for quickly computing   relatively inaccurate solutions,  which can also be used 
	for initializing a high-order method or quickly  identifying the active-set \cite{byrd2003algorithm, oberlin2006active}. Unlike second-order methods, the subproblems in first-order methods are often easier to tackle.  If only inexact subproblem solutions are required to enforce the global convergence,  the computational cost per iteration could further be reduced, which may be able to compensate for the inefficiency of the entire algorithm.  To achieve this, one must carefully handle the possible infeasibility of the subproblem constraints. Many nonlinear solvers need sway away from the main algorithm and  turn to a so-called \emph{feasibility restoration phase} to improve the constraint satisfaction.  
	In a penalty method, the penalty parameter needs to be properly tuned so that the feasibility of the nonlinear problem is not deteriorated, such penalty parameter updating strategy have been studies in \cite{ burke2018dynamic, byrd2008inexact, byrd2010infeasibility,byrd2008steering} for sequential quadratic programming methods.

	\subsection{Contributions}
	
	The major contribution of this paper is an algorithmic framework of inexact first-order penalty methods.  
	The novelties of the proposed methods include three aspects.  First,  only  inexact   solve of each subproblem is needed, which  can significantly  reduce computational effort for each subproblem.  Indeed, if the subproblem is a linear optimization problem, then only a few pivots by simplex methods can be witnessed, making the  fast computation of  relatively inaccurate solutions possible. 
	The second novelty of our proposed methods is the ability of automatic detection of potential constraint inconsistencies, so that the algorithm can automatically solve for optimality if the problem is feasible or find an {\color{blue}infeasible} stationary point if the problem is (locally) infeasible. 
	The last novel feature of our work is the worst-case complexity analysis for the proposed algorithm under {\color{blue}mild} assumptions.  We show that the KKT residuals for optimality problem and feasibility problems need at most $O(1/\epsilon^2)$ iterations  to reach below $\epsilon$ and for feasible cases the constraint violation  locally needs at most $O(1/\epsilon^2)$ iterations---a novelty rarely seen in general nonlinear constrained optimization methods.

	\subsection{Organization}
	
	In \S\ref{sec.SLP}, we describe the proposed framework of inexact first-order penalty method. 
	The global convergence and worst-case complexity analysis  of the proposed methods are analyzed in \S\ref{s.Gca}.
	Subproblems algorithms are discussed in  \S\ref{subal}. 
	Implementations of the proposed methods and the
	numerical results are discussed in \S\ref{exp}.  Finally, concluding remarks are provided in \S\ref{con}.

	\section{A Framework of Inexact First-Order Methods}\label{sec.SLP}
	
	In this section, we formulate our problem of interest and outline a framework of inexact first-order penalty  method.  
	We present our algorithm in the context of the generic nonlinear optimization problem with equality and inequality constraints
	\begin{equation}\label{prob.nlp}\tag{NLP}
	\baligned
	\min_{x\in\mathbb{R}^n} &\ f(x) & & \\
	\st &\ c_i(x) =    0\ \ \text{for all}\ \ i \in \Ecal,\\
	&\ c_i(x) \leq 0\ \ \text{for all}\ \ i \in \Ical{\color{blue},}
	\ealigned
	\end{equation}
	where the functions $f : \mathbb{R}^n \to \mathbb{R}$ and $c_i: \mathbb{R}^n \to \mathbb{R}$ for $i\in\Ecal\cup\Ical$ are continuously differentiable. 
	Our algorithms   converge to
	stationary points for the
	\emph{feasibility problem}    
	\begin{equation}\tag{FP}\label{prob.fea}
	\min_{x\in\mathbb{R}^n}  \ v(x):=\sum_{i\in\Ecal\cup\Ical} v_i(c_i(x)),
	\end{equation}
	with $v_i(z) =   \left| z \right|,\  i\in\Ecal$, where $v_i(z) =  (z)_+,\  i\in\Ical,$
	{\color{blue} and  $(\cdot)_+=\max\{\cdot{, 0}\}$}.
	It can be shown that the Clarke's generalized gradients \cite{clarke1975generalized} of $v$, denoted by $\bar \partial v(x)$, is given by  
	\[  \bar \partial v(x)  = \{ \sum_{i\in\Ecal\cup\Ical} \lambda_i \nabla c_i(x) \mid \lambda_i \in \partial_c v_i(c_i(x))\},\]
	where 
	\begin{equation}\label{fea comp}
	\partial_c v_i(c_i(x)) = \begin{cases}      
	[-1,1]  & \quad\text{if }\  i\in\Ecal\qquad  \text{ and } \  c_i(x) = 0\\  
	[0,1]   & \quad\text{if }\  i\in\Ical \qquad   \text{ and }\  c_i(x) = 0\\
	\{-1\}     & \quad\text{if }\  i\in\Ecal\qquad  \text{ and } \  c_i(x) < 0\\
	\{0\}     & \quad\text{if }\  i\in \Ical \qquad     \text{ and } \  c_i(x) < 0\\
	\{1\}     & \quad\text{if }\  i\in\Ecal\cup\Ical\  \text{ and } \  c_i(x) > 0.
	\end{cases} 
	\end{equation}
	A stationary point $x$ for \eqref{prob.fea}
	must satisfy 
	\begin{equation}\label{infeasible stationary}
	0\in\bar \partial v(x),
	\end{equation}
	and it is called an infeasible stationary point if $v(x)>0$. 
	If {\color{blue}a} minimizer of \eqref{prob.fea} violates the constraints of \eqref{prob.nlp}, then it 
	provides  a  certificate  of  {\color{blue}   local} infeasibility for \eqref{prob.nlp}. 
	Despite the possibility that problem \eqref{prob.nlp} may be 
	feasible elsewhere, it is deemed that no further progress on minimizing constraint violation locally can be made.

	If the algorithm converges to a feasible point of problem \eqref{prob.nlp},  this point should be a stationary point of the $\ell_1$ exact penalty function 
	\begin{equation}\label{penalty}
	\phi(x; \rho):= \rho f(x) + v(x),
	\end{equation}
	with $v(x)=\phi(x; 0) =  0$ for final penalty parameter $\rho>0$.  
	Such points can be characterized by the 
	$0\in \bar \partial \phi(x; \rho) = \rho \nabla f(x) + \bar \partial v(x) $, 
	which is equivalently to 
	the first-order optimality condition 
	\begin{subequations}\label{kkt opt}
		\begin{align}    \rho\nabla f(x) + \sum_{i\in\Ecal\cup\Ical} \lambda_i \nabla c_i(x)  = 0 &,\label{kkt error opt}\\
		\lambda_i\in[-1,1], i\in\Ecal;\  \lambda_i \in[0,1], i\in\Ical &, \label{kkt dual feasible opt}\\
		\sum_{c_i(x)>0}(1- \lambda_i)v_i(c_i(x))+\sum_{i\in\Ecal,c_i(x)<0}(1+\lambda_i)v_i(c_i(x))+\sum_{i\in\Ical,c_i(x)<0}\lambda_i |c_i(x)|=0 &.\label{kkt comp opt}
		\end{align}
	\end{subequations}
	A first-order stationary point $x$ for \eqref{prob.nlp} thus can be presented as a stationary point of the penalty function with $\rho>0$ and satisfying $v(x)=0$. 
	Let $\lambda=[\lambda_\Ecal^T,  \lambda_\Ical^T]^T$ be the multipliers satisfying 
	\eqref{kkt error opt}--\eqref{kkt comp opt}.
	Such a point $(x,\lambda)$ is called stationary for \eqref{prob.nlp} since  it corresponds to a Karush-Kuhn-Tucker (KKT)
	point $(x, \lambda/\rho)$ for \eqref{prob.nlp} \cite{karush2014minima, kuhn2014nonlinear}.  Also notice that \eqref{kkt opt} with $\rho=0$ can be deemed as an equivalent statement of condition \eqref{infeasible stationary}. 

	\subsection{Subproblems}
	We now describe our technique for search direction computation
	which involves the inexact solution of subproblems that are constructed using merely the first-order information of \eqref{prob.nlp}. 
	
	At the $k$th iteration, the algorithm seeks to measure the possible improvement in minimizing the linearized model \cite{pshenichnyj2012linearization} $l(d; \rho, x^k)$  of the penalty function $\phi(x; \rho)$ at $x^k$
	\begin{equation}
	l(d; \rho, x^k) := \rho \langle \nabla f(x^k),d \rangle + \sum_{i\in\Ecal\cup\Ical} v_i( c_i(x^k) +\langle \nabla c_i(x^k), d\rangle),
	\end{equation}
	where  the constant $\rho f(x^k)$ is omitted  for the ease of presentation.
	The local model $l(\cdot ; \rho, x^k)$ is convex and the subdifferential of $l(\cdot ; \rho, x^k)$ at $d$ is given by 
	\begin{equation}\label{eq.sub.kkt}
	\partial l(d ; \rho, x^k) =  \rho \nabla f(x^k) + \sum_{i\in\Ecal\cup\Ical} \nabla c_i(x^k) \partial_c v_i(c_i(x^k)+\langle \nabla c_i(x^k), d\rangle).\end{equation}
	In particular, its subdifferential at $d=0$ coincides with the Clarke's generalized subdifferential of $\phi( \cdot ; \rho)$ at $x=x^k$
	\begin{equation}\label{sub eq general}
	\partial l(0; \rho, x^k) =  \bar \partial \phi(x^k; \rho)\quad\text{and}\quad \partial l(0; 0, x^k) = \bar \partial v(x^k).
	\end{equation}
	The subproblem solver aims to find a direction $d^k$ that yields a nonnegative reduction in $l(\cdot; \rho, x^k)$ 
	and $l(\cdot; 0, x^k)$, i.e., 
	\[
	\begin{aligned}
	\Delta l (d^k; \rho, x^k) &:= l(0; \rho, x^k) - l(d^k; \rho, x^k) \ge 0\  \text{ and  }\\
	\Delta l (d^k; 0, x^k) &:= l(0; 0, x^k) - l(d^k; 0, x^k) \ge 0.
	\end{aligned}
	\]
	Such a direction  $d^k$ can be found   as an ({approximate}) minimizer of $l(\cdot; \rho_k, x^k)$ for appropriate  $\rho_k \in (0,\rho_{k-1}]$ over a convex set $X \subseteq \mathbb{R}^n$ containing $\{0\}$, i.e.,
	\begin{equation}\tag{\mbox{$\mathscr{P}$}}\label{primal prob}
	d^k := \arg\min_{d \in X}\ l(d; \rho_k, x^k)\ \ \text{for some}\ \ \rho_k \in (0,\rho_{k-1}].
	\end{equation}
	Here 
	we introduce the set $X$  for   imposing  a trust region to prevent infinite steps, which is    defined as
	$X: = \{d : \|d\| \leq \delta\}$ for  trust region radius $\delta > 0$ and  $\|\cdot\|$  
	satisfying 
	\begin{equation}\label{norm equivalent}
	\|x\|_2 \le \kappa_0 \|x\|_*,\quad \forall x\in\mathbb{R}^n
	\end{equation}
	for  constant $\kappa_0>0$ with  {\color{blue}$\|\cdot\|_*$} denoting  the dual norm of $\|\cdot\|$.  
	The most popular choice for  $\|\cdot\|$  is the $\ell_2$-norm and $\ell_\infty$-norm.   
	The latter  option   results in a linear optimization subproblem, but it is \emph{not} a requirement in our algorithm.  
	We refer to \eqref{primal prob} with $\rho>0$ as the \emph{penalty subproblem} and  \eqref{primal prob} with $\rho=0$ as the \emph{feasibility subproblem}.  In   the remainder of this paper, let $d^*(\rho, x)$ denote a minimizer of $l(d; \rho, x)$. 
	
	To alleviate the computational burden of  solving subproblems \eqref{primal prob} exactly, our algorithm accepts an inexact solution $d^k$ of 
	\eqref{primal prob} as long as it yields sufficient reduction in $l$ compared with the optimal solution of \eqref{primal prob}, i.e., 
	\begin{subequations}
		\begin{align}
		\Delta l(d^k; \rho, x^k) + \gamma_k & \ge \beta_\phi  [\Delta l(d^*(\rho,x^k); \rho, x^k)   + \gamma_k  ],\label{red cond opt}\\
		\Delta l(d^k; 0, x^k) + \gamma_k  & \ge \beta_v  [\Delta l(d^*(\rho,x^k); 0, x^k)  + \gamma_k]    \label{red cond fea}
		\end{align} 
	\end{subequations}
	where $\beta_\phi, \beta_v \in (0,1)$ with $\beta_v < \beta_\phi$ are prescribed constants and $\gamma_k \in\mathbb{R}_+$ is the relaxation error.  Further details of $\gamma_k$ will be discussed in \S\ref{sec.update}.
	
	It is often impractical to verify conditions \eqref{red cond opt} and \eqref{red cond fea}, since it requires the exact optimal solution of subproblems.
	In our algorithm, we can indirectly evaluate the optimal value  by using the dual value.  
	The Lagrangian dual of~\eqref{primal prob} is
	\begin{equation}\tag{\mbox{$\mathscr{D}$}} \label{dual prob}
	\max_{-e\le \lambda_{\Ecal} \le e,  0\le \lambda_{\Ical} \le e}\ p(\lambda; \rho, x^k):= -   \delta    \|\rho \nabla f(x^k)+\sum_{i\in\Ecal\cup\Ical}\lambda_i \nabla c_i(x^k) \|_* +  \langle c(x^k), \lambda\rangle,
	\end{equation}
	where $\lambda=[\lambda_\Ecal^T, \lambda_\Ical^T]^T$, $c(x) = [c_\Ecal(x)^T, c_\Ical(x)^T]^T$ and 
	$e$ is the vector of all 1s with appropriate dimension. If $\lambda$ is dual feasible, then   by weak duality, 
	\begin{equation}\label{weak.duality}
	p(\lambda; 0, x_k) \leq l(d; 0, x_k)\ \ \text{and}\ \ p(\lambda; \rho, x_k) \leq l(d; \rho, x_k).
	\end{equation}
	Using the dual values, we can require the direction $d^k$ to satisfy 
	\begin{subequations}
		\begin{align}
		\Delta l(d^k; \rho, x^k) + \gamma_k  & \ge \beta_\phi  [ l(0; \rho, x^k) - p(\lambda^k; \rho, x^k)  + \gamma_k ]  \label{red cond opt1}\\
		\Delta l(d^k; 0, x^k) + \gamma_k  & \ge \beta_v  [ l(0; 0, x^k) - p(\lambda^k; 0, x^k) + \gamma_k ],      \label{red cond fea1}
		\end{align} 
	\end{subequations}
	for current dual feasible estimate $\lambda^k$,  so that   
	\eqref{red cond opt} and \eqref{red cond fea} are  enforced   {\color{blue}to be} satisfied.

	An interesting aspect  to notice is that $p(\lambda; \rho, x^k)$ consists of two parts. The first term is in fact the KKT optimality residual at $(x^k, \lambda)$ scaled by the trust region radius $\delta$, while for $\lambda_i \in \partial_c v_i(c_i(x^k))$, the second  
	term $\langle c(x^k), \lambda\rangle = v(x^k)$ describes the complementarity, so that  
	the problem \eqref{dual prob} is indeed seeking to minimize the KKT residual at $x^k$.

	Before proceeding to the design of penalty parameter updating strategy, we first provide a couple of results related to our subproblems and their solutions.   

	\begin{lemma}\label{lem.subproblem}
		The following statements hold at any $x$ with $\delta > 0$.
		\begin{enumerate}
			\item[(i)]  $\Delta l ( d^*(0,x);0, x) \ge 0$ where the equality holds   if and only if $x$ is stationary for $v(\cdot)$.
			\item[(ii)]  $\Delta l ( d^*(\rho,x);\rho, x) \ge 0$ where the equality holds if and only if $x$ is stationary for $\phi(\cdot; \rho)$.  
			\item[(iii)] If $\Delta l ( d^*(\rho,x);\rho, x) = 0$ for $\rho > 0$ and $v(x)=0$, then $x$ is a KKT point for \eqref{prob.nlp}. 
		\end{enumerate}
	\end{lemma}
	
	\begin{proof}
		Note that   $\Delta l ( d^*(0,x); 0, x) = v(x) -l ( d^*(0,x); 0, x) \ge  v(x) -l (0; 0, x) =0$.  Now we investigate the {\color{blue}case where}   the equality holds: 
		\[ 
		\begin{aligned} l (0; 0, x)  = l ( d^*(0,x); 0, x) 
		& \iff 0 \text{ is stationary for } l(d; 0, x)\\
		& \iff 0 \in  {\color{blue} \partial l(d; 0, x)}\\
		& \iff 0 \in \bar \partial v(x)\\
		& \iff x \text{ is stationary for } v(x),
		\end{aligned}
		\]
		where the last equivalence is by \eqref{sub eq general}.
		This proves $(i)$.  
		
		Following the same argument for $l(d; \rho, x)$ and $\phi(x;\rho)$ using \eqref{sub eq general}, $(ii)$ holds true. 
		
		To prove $(iii)$, we know from $(ii)$ that the condition in $(iii)$ means $x$ is stationary for $\phi(x; \rho)$. Therefore, there must exist 
		$\lambda_i \in \partial_cv(c_i(x))$ such that $0=\rho \nabla f(x) + \sum_{i\in\Ecal\cup\Ical}\lambda_i \nabla c_i(x)$, which is equivalent to 
		\[   \nabla f(x) + \sum_{i\in\Ecal\cup\Ical} (\lambda_i/\rho) \nabla c_i(x) =0.\]
		Furthermore, it follows from $\lambda_i \in \partial_cv(c_i(x))$ that 
		\[ \lambda_i \ge 0, i\in \Ical, \quad \text{with}\quad \lambda_i = 0, \text{ if } c_i(x) < 0 \text{ and }  i\in\Ical,\]
		meaning the complementary condition is satisfied. 
		The constraints are all satisfied since $v(x)=0$.  Overall, we have shown that $x$ is a KKT point with multipliers 
		$\lambda_i/\rho$, $i\in\Ecal\cup\Ical$, as desired. 
	\end{proof}

	Overall, the $k$th iteration of our proposed method proceeds as in Algorithm~\ref{alg.sqo}. 
	A   direction and penalty parameter pair $(d^k,\rho_k)$ is  firstly  computed by a subproblem solver such that $d^k$ yields reductions that satisfy our conditions  \eqref{red cond opt1} and \eqref{red cond fea1}.  
	Then a line search is executed to find a step size ${\color{blue}\alpha_k}$.
	Finally, the new iterate is set as $x^{k+1} \gets x^k +  {\color{blue}\alpha_k} d^k$ and the algorithm proceeds to the next iteration.   
	The proposed first-order method for nonlinear constrained optimization, hereinafter nicknamed   \texttt{FoNCO}, is presented as  {\color{blue}Algorithm~\ref{alg.sqo}}. 
	
	\begin{algorithm}
		\caption{\emph{F}irst-\emph{o}rder methods for \emph{N}onlinear \emph{C}onstrained \emph{O}ptimization (\texttt{FoNCO})  }
		\label{alg.sqo}
		\begin{algorithmic}[1]
			\STATE Require $\{\gamma_0,   \theta_\alpha,  \beta_\alpha \} \in (0,1)$, $ \rho_{-1}\in (0,\infty)$.
			\STATE Choose $x^0 \in \mathbb{R}^n$.
			\FOR{$k \in \mathbb{N}$}
			\STATE Solve  \eqref{primal prob}   (approximately) to obtain $(d^k, \rho_k) \in \mathbb{R}^n \times (0,\rho_{k-1}]$
			\STATE \qquad or stop if a stationarity certificate is satisfied. \label{step.LP}
			\STATE Let ${\color{blue}\alpha_k}$ be the largest value in $\{ \theta_\alpha^0,  \theta_\alpha^1,  \theta_\alpha^2,  \ldots\}$ such that 
			\begin{equation}\label{LS condition} \phi(x^k; \rho_k)-\phi(x^k+ {\color{blue}\alpha_k} d^k; \rho_k) \ge \beta_\alpha {\color{blue}\alpha_k  \Delta l(d^k; \rho_k, x^k)}.
			\end{equation}
			\STATE  Set $x^{k+1} = x^k + {\color{blue}\alpha_k} d^k$, choose $\gamma_{k+1} \in (0,  \gamma_k]$ 
			\ENDFOR
		\end{algorithmic}
	\end{algorithm}
	

	\subsection{Penalty parameter update}\label{sec.update}
	
	At the $k$th iteration, the value $\rho_k \in (0,\rho_{k-1}]$ needs to be updated to ensure the direction satisfies \eqref{red cond opt1} and \eqref{red cond fea1} as described in the previous subsection can successfully be found. 
	The updating strategy for   penalty parameter $\rho$  consists of two phases.  
	The first phase   occurs within the {\color{blue}subproblem solver} and the second phase  happens after solving the subproblem. 
	
	For ease of presentation, we drop the iteration number $k$ and utilize the following shorthand notation 
	\begin{equation}\label{eq.shorthand}
	g = \nabla f(x^k),\ a^i = \nabla c_i(x^k),\ b_i = c_i(x^k),\ A = [a^1, \cdots, a^m]^T, 
	\end{equation}
	$l(d; \rho)$ and $p(\lambda; \rho)$ for the $k$th primal and dual subproblem objectives, respectively.
	Now   \eqref{primal prob} can be written as  
	\begin{equation}\tag{\mbox{$\mathscr{P}'$}}\label{primal sub}
	\min_{d \in X}\ l(d;\rho),\ \ \text{where}\ \ l(d;\rho) = \langle \rho g, d\rangle +   l(d;0),
	\end{equation}
	with its dual  \eqref{dual prob} written as 
	\begin{equation}\tag{\mbox{$\mathscr{D}'$}}\label{dual sub}
	\max_{-e\le \lambda_{\Ecal} \le e,  0\le \lambda_{\Ical} \le e}\ p(\lambda; \rho):= -  \delta\|\rho g+\sum_{i\in\Ecal\cup\Ical}\lambda_i a^i \|_* + \langle b, \lambda\rangle. 
	\end{equation}
	After dropping the iteration number $k$,   we can use  $(d_{\rho}^*,\lambda_{\rho}^*)$ to 
	represent an optimal primal-dual pair for the penalty subproblem corresponding to $\rho$; in particular, $(d_0^*,\lambda_0^*)$ represents an optimal primal-dual pair for the feasibility subproblem.  Meanwhile, we use super/sub-script $(j)$ to denote the $j$th iteration of the subproblem solver. 
	
	We are now ready to introduce the first phase of the $\rho$ updating strategy.  Suppose the subproblem solver generates  a sequence of 
	primal-dual iterates 
	$\{d^{(j)}, \lambda^{(j)}, \nu^{(j)} \}$ where $d^{(j)}$ represents the feasible solution estimate for the primal penalty subproblem, 
	$\lambda^{(j)}$ and $\nu^{(j)}$ are the dual feasible solution estimate for the dual penalty and feasibility subproblems, respectively. 
	First of all, we assume the subproblem solver finds primal solution estimate $d^{(j)}$ no worse than a trivial zero step 
	\begin{equation}\label{eq.primal_no_worse_than_zero}
	l(d^{(j)}; \rho_{(j)}) \leq l(0; \rho_{(j)}),
	\end{equation}
	the feasibility dual solution $\nu^{(j)}$ is no worse than the penalty dual solution $\lambda^{(j)}$ and both of them are no worse than 
	the initial penalty dual solution
	\begin{equation}\label{eq.dual_no_worse}
	p(\nu^{(j)}; 0) \geq p(\lambda^{(j)}; 0) \geq p(\lambda^{(0)}; 0) > - \infty.
	\end{equation}
	Both of these are reasonable assumptions.  If \eqref{eq.primal_no_worse_than_zero}
	were not to hold, 
	Lemma~\ref{lem.subproblem} indicates that this must be a certificate of stationarity for optimality or infeasibility. 
	For dual iterates, one can simply use $\nu^{(j)} = \lambda^{(j)} = \lambda^{(0)}$ if there is no better dual estimate than the initial dual estimate.

	Now, at the $j$th iteration of the subproblem solver, two ratios are defined: 
	\begin{equation}
	r_v^{(j)} := \frac{l^{(0)}_ \gamma - l(d^{(j)}; 0) }{l^{(0)}_ \gamma - (p(\nu^{(j)}; 0))_+}\ \ \text{and}\ \ r_\phi^{(j)} := \frac{l^{(0)}_ \gamma - l(d^{(j)};\rho_{(j)})}{l^{(0)}_ \gamma - p(\lambda^{(j)}; \rho_{(j)})},
	\end{equation}
	here   $l^{(0)}_ \gamma := l^{(0)} +  \gamma$
	with $ \gamma> 0$ and 
	$
	l^{(0)} := l(0; \rho) = l(0; 0) = \sum_{i\in\Ecal \cup \Ical} v_i(b_i)  = v(x^k) \geq 0
	$
	being the primal penalty and feasibility subproblem objectives at $d=0$
	for any $ \rho\ge 0$. 
	Note that the numerators of ratios $r_v^{(j)}$ and 
	$r_\phi^{(j)}$ are positive due to the presence of $\gamma$.  
	If at the $j$th iteration, we have 
	\begin{equation}\tag{\mbox{R$_v$}}\label{red.fea}
	r_v^{(j)} \ge \beta_v, 
	\end{equation}
	then the model reduction must satisfy 
	\begin{equation}\label{alter.red.constr}
	\begin{aligned}
	l_ \gamma^{(0)} - l(d^{(j)}; 0) & \geq  \beta_v(l^{(0)}_ \gamma - ( p(\nu^{(j)}; 0))_+)\\
	&\geq \beta_v(l^{(0)}_ \gamma - p(\lambda_0^*; 0)) = \beta_v(l^{(0)}_ \gamma - l(d_0^*; 0)),
	\end{aligned}
	\end{equation}
	where the first and second inequality follows by \eqref{red.fea} and the optimality of~$\lambda_0^*$ with respect to the feasibility subproblem (which is known that $p(\lambda_0^*,0)  \geq 0$) respectively, and the third one follows by strong duality. Thus condition \eqref{red cond opt} is satisfied.  
	In a similar way, if at the $j$th iteration,  we have
	\begin{equation}\tag{\mbox{R$_\phi$}}\label{red.penalty}
	r_\phi^{(j)} \geq \beta_\phi,
	\end{equation}
	then it follows that
	\begin{equation}\label{alter.red.penalty}
	\begin{aligned}
	l^{(0)}_ \gamma - l(d^{(j)}; \rho_{(j)})
	&\geq \beta_\phi(l^{(0)}_ \gamma - p(u^{(j)}; \rho_{(j)})) \\
	&\geq \beta_\phi(l^{(0)}_ \gamma - p(\lambda_{\rho_{(j)}}^*; \rho_{(j)})) = \beta_\phi(l^{(0)}_ \gamma - l(d_{\rho_{(j)}}^*; \rho_{(j)})).
	\end{aligned}
	\end{equation}
	Thus condition \eqref{red cond fea} is satisfied.

	As discussed above, the values of ratios $r_v^{(j)}$ and $r_\phi^{(j)}$ reflect the inexactness of  current primal-dual $\{d^{(j)}, \lambda^{(j)}, \nu^{(j)} \}$. We need another ratio to measure the satisfaction of the  complementarity.  Define the index sets
	\[ \begin{aligned}
	\Ecal_+(d) & := \{ i \in \Ecal : \langle  a^i, d\rangle +  b_i > 0 \},\\
	\Ecal_- (d) & := \{ i\in  \Ecal : \langle  a^i, d\rangle +  b_i < 0 \},\\	
	\text{and}\ \ \Ical_+(d)  & := \{ i\in \Ical : \langle  a^i, d\rangle +  b_i > 0\}.
	\end{aligned}
	\]
	The complementarity measure can be defined accordingly:
	\[\chi(d,\lambda):=
	\sum_{i\in\Ecal_+(d)\cup\Ical_+(d)} (1- \lambda_i) v_i(\langle a^i, d\rangle + b_i)  + \sum_{i\in\Ecal_-(d)}( 1+\lambda_i)  v_i(\langle a^i, d\rangle + b_i).
	\]
	With an optimal primal-dual solution $(d_\rho^*, \lambda_\rho^*)$ for a penalty subproblem, one has $\lambda_i(\lambda_\rho^*) = 1$ for $i \in \Ecal_+(d_\rho^*)$, $\lambda_i(\lambda_\rho^*) = -1$ for $i \in \Ecal_-(d_\rho^*)$, and $\lambda_i(\lambda_\rho^*) = 1$ for $i \in \Ical_+(d_\rho^*)$, from which it follows that $\chi(d_\rho^*,\lambda_\rho^*)=0$.
	For an inexact solution, we   require that $(d^{(j)}, \lambda^{(j)})$ satisfies
	\[ \chi^{(j)} := \chi(d^{(j)},\lambda^{(j)})  \le (1 - \beta_v)^2l^{(0)}_ \gamma,\]
	or, equivalently, 
	\begin{equation}\tag{\mbox{R$_c$}}\label{red.comp}
	r_c^{(j)} := 1- \sqrt{\frac{ \chi^{(j)}}{l^{(0)}_ \gamma}} \ge \beta_v.
	\end{equation}
	Note that the numerator in $r_c^{(j)}$ is always positive due to the presence of $\gamma > 0$.
	{\color{blue}
		Therefore, for a given $ \gamma \in (0,\infty)$, \eqref{red.comp} will hold for sufficiently accurate primal-dual solutions of the penalty subproblem.  
	}

	For a fixed $\rho$,  conditions  \eqref{red.penalty} and  \eqref{red.comp} will eventually be satisfied as the subproblem algorithm proceeds.  However, this may not be the case for condition \eqref{red.fea}. When this happens,  $d^k$ is deemed to be a   ``successful'' inexact direction for minimizing the penalty function, but an ``unsuccessful'' direction for improving feasibility. The intuition underlying this phenomenon is that a large penalty parameter places too much emphasis on the objective function---a reason to reduce the penalty parameter. 
	Thus we can update the parameter while solving the subproblem as follows. 
	Given 
	\begin{equation}\label{eq.parameters}
	0 < \beta_v < \beta_\phi < 1,
	\end{equation}
	we initialize $\rho_{(0)} \gets \rho_{k-1}$ 
	(from the preceding iteration of the outer iteration) and apply the subproblem solver to \eqref{primal sub} to generate $\{(d^{(j)},\lambda^{(j)},\nu^{(j)})\}$. We continue to iterate toward solving \eqref{primal sub} 
	until  \eqref{red.penalty} and \eqref{red.comp} are satisfied. Then we terminate the subproblem algorithm if \eqref{red.fea} is also satisfied or reduce the penalty parameter by setting
	\begin{equation}\label{update.rho}
	\rho_{(j+1)} \gets \theta_\rho \rho_{(j)}
	\end{equation}
	for some prescribed $\theta_\rho\in (0,1)$.

	
	It is possible that \eqref{red.penalty}, \eqref{red.comp}, and \eqref{red.fea} all hold with $d^{(j)} = 0$ causing the subproblem solver takes a null step. In such a case, {\color{blue} we have the subproblem solver terminate} with $d^{(j)} = 0$, causing the outer iteration  to take a null step in the primal space.  This would be followed by a decrease in $ \gamma$, prompting  the outer iteration  to eventually make further progress through solving the subproblem or terminate with a stationarity certificate.  
	
	On the other hand, if $x^k$ is not stationary with respect to $\phi(\cdot,\rho)$ for any $\rho \in (0,\rho_{k-1}]$, but is stationary with respect to $v$, then for $(d_0^*,\lambda_0^*)$ one has
	\[
	\frac{l^{(0)}_ \gamma - l(d_0^*; 0)}{l^{(0)}_ \gamma - (p(\lambda_0^*; 0))_+} = \frac{ \gamma}{ \gamma} = 1,
	\]
	meaning that $r_v^{(j)} > \beta_v$ for $(d^{(j)},\nu^{(j)})$ in a neighborhood of $(d_0^*,\lambda_0^*)$.  One should expect that \eqref{dust} would only reduce the penalty parameter a finite number of times during one {\color{blue}subproblem solver}.  If $x^k$ is near an infeasible stationary point, this may happen consecutively for many subproblems. This may quickly drive the penalty to 0, leading to an infeasible stationary point.

	{\color{blue}Now we introduce the  dynamic updating strategy (DUST) \cite{burke2018dynamic} stated as:
		\begin{equation}\tag{DUST}\label{dust}
		\boxed{ 
			\baligned
			& \text{\small Given $\rho_{(j)}$ and the $j$th iterate $(d^{(j)},\lambda^{(j)},\nu^{(j)})$, perform the following:} \\
			& \text{\small \qquad $\bullet$ if 
				\eqref{red.penalty}, \eqref{red.comp}, and \eqref{red.fea} hold, then terminate;} \\
			& \text{\small \qquad $\bullet$ else if \eqref{red.penalty} and \eqref{red.comp} hold, but \eqref{red.fea} does not, then apply \eqref{update.rho};} \\
			& \text{\small \qquad $\bullet$ else set $\rho_{(j+1)} \gets \rho_{(j)}$.}
			\ealigned
		}
		\end{equation}}
	
	After solving the subproblem, we consider the second phase of the $\rho$ updating strategy.  Let $\tilde\rho_k$ be the value of the penalty parameter obtained by applying \eqref{dust} within the $k$th {\color{blue}subproblem solver}.  Then, given a constant $\beta_l \in (0,\beta_\phi(1-\beta_v)]$, we require $\rho_k \in (0,\tilde\rho_k]$ so that
	\begin{equation}\label{dust.after}
	\Delta l(d^k; \rho_k, x^k) +  \gamma_k \geq \beta_l (\Delta l(d^k; 0, x^k)+ \gamma_k),
	\end{equation}
	where the right-hand side of this inequality is guaranteed to be positive due to \eqref{red.fea}.  This can be guaranteed by the following  \emph{P}osterior \emph{S}ubproblem \emph{ST}rategy:
	\begin{equation}\tag{PSST}\label{psst}
	\boxed{
		\rho_k \gets \begin{cases} \tilde\rho_k  & \text{if this yields \eqref{dust.after}}  \\
		\cfrac{(1-\beta_l) ( \Delta l(d^k; 0, x^k) +  \gamma_k) }{\langle \nabla f(x^k), d^k\rangle   } & \text{otherwise.}
		\end{cases}
	}
	\end{equation}
	It is possible that $\langle \nabla f(x^k), d^k\rangle \le 0$, implying 
	\[\Delta l(d^k; \tilde\rho_k, x^k) = - \langle \tilde\rho_k\nabla f(x^k), d^k\rangle  +  \Delta l(d^k; 0, x^k) \ge \Delta l(d^k; 0, x^k),\]
	so that \eqref{dust.after} is always true by setting $\rho_k = \tilde\rho_k$. Thus the denominator in the latter formula \eqref{psst} is always positive. 
	On the other hand, if the choice $\rho_k = \tilde\rho_k$ does not yield \eqref{dust.after}, then, by setting $\rho_k$ according to the latter formula in \eqref{psst}, it follows   that
	\[ \rho_k\langle \nabla f(x^k), d^k\rangle \le (1-\beta_l)(\Delta l(d^k; 0, x^k) +  \gamma_k),\]
	which means that
	\[ \Delta l(d^k; \rho_k, x^k) +  \gamma_k =  \Delta l(d^k;0, x^k) - \rho_k\langle \nabla f(x^k), d^k\rangle +  \gamma_k \ge \beta_l(\Delta l(d^k; 0,  x^k) +  \gamma_k),\]
	implying that \eqref{dust.after} holds.
	This idea is similar to the updating strategy in    various algorithms  employing a merit function such as  \cite{burke2014sequential}.  The difference  is that this model reduction condition is imposed inexactly (due to the presence of $ \gamma_k > 0$), making \eqref{psst} more suitable for an inexact  algorithmic framework.

	We  summarize the framework of a subproblem solver employing the \eqref{dust} and \eqref{psst} in Algorithm~\ref{alg.subproblem}.

	\begin{algorithm}[H]
		\caption{A Framework of Subproblem Algorithm for Solving \eqref{primal prob}. }
		\label{alg.subproblem}
		\begin{algorithmic}[1]
			\STATE Require $(\gamma_k,\beta_\phi) \in (0,1)$, $\beta_v\in (0,\beta_\phi)$,  $\beta_l \in (0,\beta_\phi(1-\beta_v))$ and $(\rho_{-1}, \gamma_0) \in (0,\infty)$
			\STATE Set $\rho_{(0)} \gets \rho_{k-1}$
			\FOR{$j \in \mathbb{N}$} 
			
			\STATE Generate a primal-dual feasible  solution estimate $(d^{(j)}, \lambda^{(j)}, \nu^{(j)})$
			\STATE Set $\rho_{(j+1)}$ by applying \eqref{dust}
			\ENDFOR
			\STATE  Set $d^k \gets d^{(j)}$ and $\tilde \rho_k \gets \rho_{(j)}$.
			\STATE Set $\rho_k$ by applying \eqref{psst}
			
		\end{algorithmic}
	\end{algorithm}

	\section{Convergence analysis} \label{s.Gca}
	
	In this section, we analyze the behavior of our proposed algorithmic framework. We first show that   if \eqref{dust} is employed within an algorithm for solving \eqref{primal sub}, then, under reasonable assumptions on the subproblem data, it is only triggered finite number of times.  
	The second part of this section focuses on the global convergence, which {\color{blue}shows} that our proposed algorithm will either converge to a stationary point if \eqref{prob.nlp} is feasible  or an infeasible stationary point if \eqref{prob.nlp} is  {\color{blue}locally} infeasible under general assumptions. 
	
	One of the contributions in this paper is the complexity analysis for the proposed method.  
	We  derive the worst-case complexity analysis of the KKT residuals for both the nonlinear optimization problem \eqref{prob.nlp}  and the feasibility problem \eqref{prob.fea}.  Local complexity analysis  for   constraint violation is also proved at the last part of this section.

	\subsection{Worse-case complexity for a single subproblem}\label{sec.bound.rhoj}
	
	The goal of this subsection is to show  that the subproblem solver terminates after reducing the penalty parameter for a finite number of times by employing the \eqref{dust} within the subproblem solver for solving \eqref{primal sub}.  Specifically,  we can show that there exists a sufficiently small $\tilde\rho$ such that 
	for any   $\rho\in(0,\tilde\rho]$, if \eqref{red.penalty} and \eqref{red.comp} are satisfied, then \eqref{red.fea} is also satisfied---a criterion that \eqref{dust} will not be triggered and the subproblem solver should be terminated at this moment.   
	Our complete subproblem algorithm utilizing  \eqref{dust} and \eqref{psst} is stated  as Algorithm~\ref{alg.subproblem}.  It should be clear that in the inner loop (over $j$) one is solving a subproblem with quantities dependent on the $k$th iterate of the main algorithm; see \eqref{eq.shorthand}.

	The assumption needed for this analysis is simply the primal-dual feasibility of the iterates, which is formulated as the following. 
	
	\begin{assumption}\label{ass.subproblem}
		For all $j \in \mathbb{N}$, the sequence $\{(d^{(j)},\lambda^{(j)}), \nu^{(j)}\}$ has $d^{(j)} \in X$, $\lambda^{(j)}$ and $\nu^{(j)}$ are feasible for \eqref{dual sub},  and \eqref{eq.primal_no_worse_than_zero}   and \eqref{eq.dual_no_worse} hold.
	\end{assumption}

	We first show that the differences between the primal and dual values of the optimality and feasibility subproblems are bounded with respect to $\rho$. Therefore, as $\rho$ tends sufficiently small, the optimality primal (dual) subproblem will approaches the feasibility primal (dual) subproblem.   
	
	\begin{lemma}\label{lem primal dual value}
		Under Assumptions~\ref{ass.subproblem}, it follows that, for any $j\in \mathbb{N}$,
		\begin{subequations}
			\begin{align}
			| l (d^{(j)}; \rho_{(j)}) - l (d^{(j)}; 0) | & \le \kappa_2 \rho_{(j)} \label{J.krho}\\ \text{and}\ \ 
			|p(\lambda^{(j)}; \rho_{(j)}) - p(\lambda^{(j)}; 0) | & \le \kappa_3 \rho_{(j)},\label{D.krho}
			\end{align}
		\end{subequations}
		with $\kappa_2:= \kappa_0\delta\|g\|_2$ and $\kappa_3=\delta\|g\|$. In particular, $\kappa_2 = \delta\|g\|_2$ if $\|\cdot\| = \|\cdot\|_2$ and 
		$\kappa_2 = \sqrt{n} \delta \|g\|_2$ if $\|\cdot\| = \|\cdot\|_\infty$. 
	\end{lemma}
	
	\begin{proof}
		For the primal values,  it holds true that   
		\begin{equation*}
		|l(d^{(j)}; \rho_{(j)}) - l(d^{(j)}; 0)|  = |\rho_{(j)} \langle g, d^{(j)}\rangle  |  \le \rho_{(j)} \|g\|_2  \| d^{(j)} \|_2  \le \rho_{(j)} \kappa_0\|g\|_2  \|d^{(j)} \|,
		\end{equation*}
		where the second inequality is from the requirement \eqref{norm equivalent}.
		It then follows that   
		\[     |l(d^{(j)}; \rho_{(j)}) - l(d^{(j)}; 0)|    \le \rho_{(j)} \kappa_0\delta\|g\|_2 = \kappa_2 \rho_{(j)},
		\]
		proving \eqref{J.krho}. 
		
		The difference between the dual values   is   given by 
		\begin{equation*} 
		\begin{aligned}
		&| p (\lambda^{(j)}; \rho_{(j)}) - p (\lambda^{(j)}; 0) | \\
		=\ & | - \delta \| \rho_{(j)} g+\sum_{i\in\Ecal\cup\Ical}\lambda_i a_i  \| +   \delta \| \sum_{i\in\Ecal\cup\Ical}\lambda_i a_i  \| |,\\
		\le \ &     \delta  \| \sum_{i\in\Ecal\cup\Ical}\lambda_i a_i  - (\rho_{(j)} g+\sum_{i\in\Ecal\cup\Ical}\lambda_i a_i ) \|,\\
		= \ &      \rho_{(j)} \delta\|g\|,
		\end{aligned}
		\end{equation*}
		completing the proof of \eqref{D.krho}.
	\end{proof}

	Now we are ready to prove our main result in this section, which needs the following definition    
	\[
	\mathcal{R} = \{j : (d^{(j)},\lambda^{(j)})\ \text{satisfies}\ \eqref{red.penalty}\text{ and }\eqref{red.comp}\ \text{but not } \eqref{red.fea}\},
	\]
	meaning that $\mathcal{R}$ is the set of subproblem iterations in which \eqref{update.rho} is triggered.   

	\begin{theorem}\label{thm.subproblem}
		Suppose Assumptions~\ref{ass.subproblem} holds and let 
		\[
		\kappa_4 := \inf_{j\in\mathcal{R}}\{ l^{(0)} -l(d^{(j)}; \rho_{(j)})\} \ge 0\ \  \text{and}\ \ \kappa_5 := \inf_{j\in\mathcal{R}} \{l^{(0)}-p(\lambda^{(j)}; 0)\} \ge 0.
		\]
		Then we have the following two cases. 
		\begin{enumerate}
			\item[(i)]  If $g=0$, then \eqref{dust} is never triggered during the {\color{blue}subproblem solver}.  
			\item[(ii)] If $g\ne0$, for $\rho_{(j)}\in(0,\hat\rho]$, where 
			\begin{equation}\label{hat.rho}
			\hat \rho : = \frac{ \gamma + \min\{\kappa_4, \kappa_5\}}{\max\{\kappa_2,\kappa_3\}}\left( 1-\sqrt{\beta_v/\beta_\phi}\right),
			\end{equation}
			if $(d^{(j)},\lambda^{(j)})$ satisfies \eqref{red.penalty} and \eqref{red.comp}, then $(d^{(j)},\nu^{(j)})$ satisfies \eqref{red.fea}.  In other words, for any $\rho_{(j)}\in(0,\hat\rho]$, the update \eqref{update.rho} is never triggered by \eqref{dust}. 
		\end{enumerate}
	\end{theorem}
	
	\begin{proof}  
		We first prove $(i)$.  
		If $g=0$, we know that $l(d^{(j)}; \rho) = l(d^{(j)}; 0)$ and that {\color{blue} $p(\nu^{(j)}; 0) \ge p(\lambda^{(j)}; 0) = p(\lambda^{(j)}; \rho)$ by the selection 
			of $\nu^{(j)}$, implying $r_v^{(j)} \ge r_\phi^{(j)}$. }
		Therefore, if \eqref{red.penalty} and \eqref{red.comp} are satisfied, then  {\color{blue} $r_v^{(j)} \ge r_\phi^{(j)} \ge \beta_\phi > \beta_v$} satisfying \eqref{red.fea}, as desired.

		As for $(ii)$, the denominator of $\tilde\rho$ is positive since $g\neq 0$.  We prove $(ii)$ by contradiction and assume        that $\mathcal{R}$ is infinite, indicating that  the subproblem solver is never terminated and $\rho_{(j)}$ is reduced infinite many times and driven to 0.   We have from \eqref{J.krho} that 
		\[
		-\kappa_2\rho_{(j)} \le  l(d^{(j)}; \rho_{(j)})  -  l(d^{(j)}; 0)  \le \kappa_2\rho_{(j)}\ \ \ \text{for any}\  j \in \mathcal{R},
		\]
		which, after adding and dividing through by $l_ \gamma^{(0)} - l(d^{(j)}; \rho_{(j)})$, yields for $j \in \mathcal{R}$ that
		\begin{equation}\label{lele2}
		1 -\frac{\kappa_2\rho_{(j)}}{l^{(0)}_ \gamma-l(d^{(j)}; \rho_{(j)})}  \le \frac{l^{(0)}_ \gamma - l (d^{(j)}; 0)}{l^{(0)}_ \gamma-l(d^{(j)}; \rho_{(j)})} \le 1+\frac{\kappa_2\rho_{(j)}}{l^{(0)}_ \gamma-l(d^{(j)}; \rho_{(j)})}.
		\end{equation}
		Thus, for any
		\[
		\rho_{(j)} \le \frac{  \gamma + \kappa_4}{\kappa_2}\left(1-\sqrt{\frac{\beta_v}{\beta_\phi}}  \right) \leq \frac{ l^{(0)}_ \gamma - l(d^{(j)}; \rho_{(j)})}{\kappa_2}\left(1-\sqrt{\frac{\beta_v}{\beta_\phi}}\right),
		\]
		it follows from the first inequality of \eqref{lele2} that
		\begin{equation}\label{satisfy1}
		\frac{l ^{(0)}_ \gamma-l (d^{(j)}; 0)}{l ^{(0)}_ \gamma-l (d^{(j)} ; \rho_{(j)})} \geq \sqrt{\frac{\beta_v}{\beta_\phi}}.
		\end{equation}
		Following a similar argument from \eqref{D.krho}, it follows that for any
		\[\rho_{(j)} \le  \frac{ \gamma + \kappa_5}{\kappa_3}\left( 1-\sqrt{\frac{\beta_v}{\beta_\phi}}\right) \le  \frac{l^{(0)}_ \gamma - p(\lambda^{(j)}; 0)}{\kappa_3}\left( 1-\sqrt{\frac{\beta_v}{\beta_\phi}}\right),
		\]
		one finds that
		\begin{equation}\label{satisfy2} 
		\frac{l^{(0)}_ \gamma-p (\lambda^{(j)}; \rho_{(j)})}{l^{(0)}_ \gamma-p (\lambda^{(j)}; 0)} \geq \sqrt{\frac{\beta_v}{\beta_\phi}}. 
		\end{equation}
		Overall, we have shown that for any $\rho_{(j)} \le \tilde\rho$ with $\tilde\rho$ defined in \eqref{hat.rho}, it follows that 
		\eqref{satisfy1} and \eqref{satisfy2} both hold true. 
		
		Since our supposition that $\mathcal{R}$ is infinite implies that $\rho_{(j)} \to 0$, we may now proceed under the assumption that $j \in \mathcal{R}$ with $\rho_{(j)} \in (0,\tilde\rho]$.  Let us now define the ratios
		\[
		\bar r_v^{(j)} := \frac{l^{(0)}_ \gamma - l(d^{(j)}; 0) }{l^{(0)}_ \gamma -  p(\lambda^{(j)}; 0)   },  \]
		where it must be true that $r_v^{(j)} \ge \bar r_v^{(j)} $ by the definition of operator $(\cdot)_+$.  From \eqref{satisfy1} 
		\[
		\frac{\bar r^{(j)}_v}{r^{(j)}_\phi} 
		=   \frac{l^{(0)}_ \gamma-l(d^{(j)}; 0)}{l^{(0)}_ \gamma-l(d^{(j)}; \rho_{(j)})} \frac{l^{(0)}_ \gamma-p ( \lambda^{(j)}; \rho_{(j)})}{l^{(0)}_ \gamma-p ( \lambda^{(j)}; 0)}  
		\ge \frac{\beta_v}{\beta_\phi},
		\]
		yielding 
		\[ r^{(j)}_v \ge \bar r^{(j)}_v \ge \frac{\beta_v}{\beta_\phi} r^{(j)}_\phi \ge \beta_v. \]
		However, this contradicts the fact that $j\in\mathcal{R}$.  Overall, since we have reached a contradiction, we may conclude that $\mathcal{R}$ is finite. 
	\end{proof}
	
	We can use Theorem~\ref{thm.subproblem}  to estimate the number of reductions occured during a single {\color{blue}subproblem solver}, as well as a lower bound of the penalty parameter after solving the subproblem, which is summarized in the following theorem. Since it describes results about the 
	main algorithm, we add back the $k$ index to denote the $k$th iteration of main algorithm. 
	
	\begin{theorem} \label{thm.subproblem.compleixty}
		Suppose Assumptions~\ref{ass.subproblem} holds, then after solving the $k$th subproblem, we have 
		\begin{equation}\label{rho lower bound} \tilde\rho_k \ge \min\left(\rho_{k-1}, \tfrac{ \theta_\rho\gamma_k}{ \max(\kappa_0^2,1)\delta }\left( 1-\sqrt{\tfrac{\beta_v}{\beta_\phi}}\right) \tfrac{1}{\|\nabla f(x^k)\| } \right).\end{equation}
		Moreover, \eqref{dust} is triggered at most 
		\begin{equation}\label{how many updates}
		\left\lceil \tfrac{1}{\ln\theta_\rho} \ln\left(\tfrac{   \gamma_k}{ \max(\kappa_0^2,1)\delta }\left( 1-\sqrt{\tfrac{\beta_v}{\beta_\phi}}\right) \tfrac{1}{ \rho_{k-1} \|\nabla f(x^k)\| }\right) \right\rceil 
		\end{equation}
		times during solving the $k$th subproblem.
		
	\end{theorem}

	\begin{proof}
		From Theorem~\ref{thm.subproblem}, we know that if $\hat\rho_k\ge \rho_{k-1}$, then \eqref{dust} is never triggered. 
		If this is not the case, since $\rho$ is reduced by a fraction whenever it is updated, from Theorem~\ref{thm.subproblem}, we know the final $\rho$ returned by the subproblem solver  must satisfy 
		\[ \tilde\rho_k \ge \theta_\rho\hat\rho_k \ge \frac{ \theta_\rho\gamma_k}{\max\{\kappa_0^2\delta\|\nabla f(x^k)\|,\delta\|\nabla f(x^k)\|\}}\left( 1-\sqrt{\beta_v/\beta_\phi}\right).\]
		by noticing 
		\[ \|\nabla f(x^k)\|_2 \le \kappa_0\|\nabla f(x^k)\|.
		\]
		This completes the proof of  \eqref{rho lower bound}.
		
		For \eqref{how many updates},  suppose  $\hat\rho_k  <  \rho_{k-1} $ so that \eqref{dust} is triggered during the {\color{blue}subproblem solver}. Also suppose after $\hat j$ reductions, we have $ \theta_\rho^{\hat j}  \rho_{k-1} \le  \hat\rho_k $. Taking logarithm of both sides, after simple rearrangement, we have 
		\[ \hat j \le \frac{\ln(\hat\rho_k/\rho_{k-1})}{\ln\theta_\rho}. \]
		Notice that both the denominator and the numerator are negative.  This inequality, combined with 
		\eqref{rho lower bound}, proves \eqref{how many updates}.
	\end{proof}
	
	From  \eqref{rho lower bound} and \eqref{how many updates} in  Theorem~\ref{thm.subproblem.compleixty}, 
	it would be worth noticing that many factors may affect the 
	the number of times that \eqref{dust} is triggered within a single subproblem. 
	\begin{itemize}
		\item Smaller  $\|\nabla f(x^k)\|$  will result in fewer  \eqref{dust} updates.  Intuitively,  in  this case, $l(d; \rho, x^k)$  is close to $l(d; 0, x^k)$, 
		so that any  direction successful for optimality may be also successful for feasibility.  As for larger $\|\nabla f(x^k)\|$, we may need more updates. 
		\item The accuracy tolerance $\gamma_k$ also affects the number of updates needed.  If we have aggressively small $\gamma_k$, meaning the subproblem needs to be solved more accurately, then \eqref{dust} updates may happen more frequently. 
		\item The trust region radius also plays a role in the number of \eqref{dust} updates, and smaller $\delta$ may lead to fewer updates.  This is reasonable since the difference between  $l(d; \rho, x^k)$ and  $l(d; 0, x^k)$ should be smaller in this case within trust region $X$. 
		\item We can also see the influence of the algorithmic parameters here.  A smaller $\theta_\rho$ naturally leads to fewer updates but possibly smaller $\rho$ since it is reduced more aggressively each time.  It would be interesting to see that if one chooses $\beta_v \to \beta_\phi$ (note $\beta_v < \beta_\phi$), 
		then \eqref{dust} may occur a lot more times. The intuition of this case is that we require a  direction successful for optimality should also be the same successful for feasibility, which could only happen for very small $\rho$.
	\end{itemize}

	
	\subsection{Global Convergence}

	In this subsection,  we show that if \eqref{dust} and \eqref{psst} are used to solve \eqref{prob.nlp} in a {\color{blue}penalty-SLP algorithm, then the algorithm can converge from any starting point using reasonable assumptions.} Specifically, if \eqref{dust} and \eqref{psst} are only triggered a finite number of times, then every limit point of the iterates is either infeasible stationary or first-order stationary for \eqref{prob.nlp}.  Otherwise, if \eqref{dust} and \eqref{psst} are triggered an infinite number of times, driving the penalty parameter to zero, then every limit point of the iterates is either an infeasible stationary point or a feasible point at which a constraint qualification fails to hold.
	
	For the analysis in this section, we extend our use of the sub/superscript $k$ to represent the value of quantities associated with iteration $k \in \mathbb{N}$.  For instance, $\mathcal{R}^k$ denotes the set $\mathcal{R}$ defined in \S\ref{sec.bound.rhoj} at the $k$th iteration.

	In the whole process of analysis, we assume the following.
	
	\bassumption\label{global} Functions    
	$f$ and $c_i$ for all $i \in \Ecal\cup\Ical$, and their first- and second-order  derivatives, are all  bounded in an open convex set containing $\{x^k\}$ and $\{x^k+d^k\}$.  Also assume that $\gamma_k\to0$. 
	\eassumption
	
	
	Define the index set 
	\[
	\mathcal{U} := \{k \in \mathbb{N} : \mathcal{R}^k\ne\emptyset \}.
	\]
	Moreover, for every $k\in\mathcal{U}$, let $j_k$ be the subproblem iteration number corresponding to the value of the smallest ratio $r_v$, i.e.,
	\[   r_v^{(j_k)} \le r_v^{(i_k)} \quad \text{for any }\  i_k\in\mathcal{R}^k.\]
	Also, define the index set 
	\[ 
	\mathcal{T} := \{ k \in\mathbb{N}: \rho_k \  \text{is reduced by}\ \eqref{psst} \}.
	\]
	From these definitions, it follows that $\rho_k < \rho_{k-1}$ if and only if $k\in\mathcal{U}\cup\mathcal{T}$.

	We also have the following fact about the subproblem solutions, the proof of which is skipped here since it can be easily derived by noticing $\|d^k\|_2 \le \kappa_0\|d^k\| \le \kappa_0\delta$ in the proof of \cite[Lemma 10]{burke2018dynamic}.
	\begin{lemma}\label{lem.trustregion}
		Under Assumption~\ref{ass.subproblem} and \ref{global}, it follows that, for all $k\in\mathbb{N}$, 
		the stepsize satisfies 
		\[ {\color{blue}\alpha_k} \ge \frac{\theta_\alpha(1-\beta_\alpha)}{\delta\kappa_0\kappa_1}  \Delta l(d^k; \rho_k, x^k).\]
	\end{lemma}
	
	%
	%

	We now prove that, in the limit, the reductions in the models of the constraints violation measure and the penalty function vanish. For this purpose, it will be convenient to work with the shifted penalty function
	\[ \varphi(x,\rho) := \rho ( f(x) - \underline f)  + v(x) \ge 0, \]
	where $\underline f$(its existance follows from 
	Assumption~\ref{global}) is the infimum of $f$ over the smallest convex set containing $\{x^k\}$. In the following lemma, it proves that the function $\varphi$ possesses a useful monotonicaity property.

	\begin{lemma}\label{lem.varphi}
		Under Assumption~\ref{ass.subproblem} and \ref{global}, it holds that, for all $k \in \mathbb{N}$,
		\begin{align}
		\varphi(x^{k+1}, \rho_{k+1})   \le \varphi(x^k,\rho_k)  - \frac{\theta_\alpha(1-\beta_\alpha)\beta_\alpha}{\delta\kappa_0\kappa_1} [ \Delta l(d^k; \rho_k, x^k)]^2,  \label{var.red.rho} 
		\end{align}
		
	\end{lemma}

	
	\begin{proof}
		From the line search condition \eqref{LS condition}  
		\[ \varphi(x^{k+1}, \rho_{k+1}) \le \varphi(x^k,\rho_k) - (\rho_k-\rho_{k+1})(f(x^{k+1}) - \underline f) -  \beta_\alpha  {\color{blue}\alpha_k}  \Delta l(d^k; \rho_k, x^k).\]
		Then \eqref{var.red.rho} follows from this inequality, Lemma~\ref{lem.trustregion},  the fact that $\{\rho_k\}$ is monotonically decreasing, and  $f(x^{k+1}) \ge \underline f$ for 
		all $k\in\mathbb{N}$. 
	\end{proof}

	We now show the  model reductions and duality gap all vanish asymptotically.
	
	\begin{lemma}\label{lem.delta_2_zero}
		Under Assumption~\ref{ass.subproblem} and \ref{global}, the following limits hold.
		\begin{enumerate}
			\item[(i)] $0 =  \lim\limits_{k\to \infty} \Delta l(d^k; \rho_k, x^k)=   \lim\limits_{k\to\infty} \Delta l(d^k; 0, x^k)$,
			\item[(ii)] $0 = \lim\limits_{k\to\infty} [l(0; \rho_k, x^k ) - p(\lambda^k; \rho_k, x^k )]  = \lim\limits_{k\to\infty} [l(0; 0, x^k ) - p(\nu^k; 0, x^k )]$,
			\item[(iii)] $0 =  \lim\limits_{k\to\infty} \Delta l(d^*(0,x^k); 0, x^k )=  \lim\limits_{k\to\infty} \Delta l(d^*(\rho_k, x^k); \rho_k, x^k )$,
			\item[(iv)] $0 =  \Delta l (d^*(0, x^* ); 0, x^*) =    \Delta l  (d^*(\rho_*, x^* ); \rho_*, x^*)$ with $\rho_* := \lim\limits_{k\to\infty} \rho_k$  for   any limit point 
			$x^*$ of $\{x^k\}$. 
		\end{enumerate}
	\end{lemma}
	
	\begin{proof}
		Let us first prove $(i)$ by contradiction.  Suppose that $\Delta l(d^k; \rho_k, x^k)$ does not converge to 0. Then, there exists a constant 
		$\epsilon>0$ and an infinite $\mathcal{K} \subseteq \mathbb{N}$ such that $\Delta l(d^k; \rho_k, x^k) \geq \epsilon$ for all $k\in \mathcal{K}$. It then follows from Lemma~  \ref{lem.varphi} that $\varphi(x^k; \rho_k) \to -\infty$, which contradicts the fact that $\{\varphi(x^k,\rho_k)\}$ is bounded below by zero. Therefore, $\Delta l(d^k; \rho_k, x^k)\to 0$.  The second limit in $(i)$ follows from \eqref{dust.after} and $\gamma_k\to0$. 
		
		The limits in $(ii)$ and $(iii)$ follow directly from the limits in $(i)$ and the inequalities in \eqref{alter.red.constr} and \eqref{alter.red.penalty} along with $\gamma_k \to 0$; $(iv)$ follows directly from $(iii)$. 
	\end{proof}
	
	We now provide our first global convergence theorem.

	\begin{theorem}\label{thm.global1}
		Under Assumption~\ref{ass.subproblem} and \ref{global}, the following statements hold.  
		\begin{enumerate}
			\item[(i)] Any limit point of $\{x^k\}$ is first-order stationary for $v$, i.e., it is feasible or an infeasible stationary point for \eqref{prob.nlp}. 
			\item[(ii)] If $\rho_k \to \rho_*$ for some constant $\rho_*>0$ and $v(x^k)\to0$, then any limit point $x^*$ of $\{x^k\}$ with $v(x^*)=0$ is a KKT point for \eqref{prob.nlp}. 
			\item[(iii)] If $\rho_k \to 0$, then either all limit points of $\{x^k\}$ are feasible for \eqref{prob.nlp} or all are infeasible. 
		\end{enumerate}
	\end{theorem}
	\begin{proof}
		Part $(i)$ and part $(ii)$ follow by combining Lemma~\ref{lem.delta_2_zero}$(iv)$ with Lemma~\ref{lem.subproblem}$(i)$ and Lemma~\ref{lem.delta_2_zero}$(iv)$ with {\color{blue}Lemma~\ref{lem.subproblem}$(iii)$} respectively.
		
		We prove $(iii)$ by contradiction.  Suppose there exist infinite $\mathcal{K}^* \subseteq \mathbb{N}$ and $\mathcal{K}^\times  \subseteq \mathbb{N}$ such that $\{x^k\}_{k\in\mathcal{K}^*}\to x^*$ with $v(x^*)=0$ and 
		$\{x^k\}_{k\in\mathcal{K}^\times} \to x^\times$ with $v(x^\times) = \epsilon > 0$.  Since $\rho_k \to 0$, there exists $k^* \ge 0$ such that for all $k\in\mathcal{K}^*$ and $k\ge k^*$ one has that 
		$\rho_k(f(x^k)-\underline f) < \epsilon/4$ and $v(x^k) < \epsilon/4$, meaning that $\varphi(x^k, \rho_k) < \epsilon/2$.  On the other hand, it follows that $\rho_k(f(x^k)-\underline f) \ge 0$ for all $k \in \mathbb{N}$ and there exists $k^\times \in \mathbb{N}$ such that $v(x^k)\ge \epsilon/2$ for all $k\ge k^\times$ with $k \in \mathcal{K}^\times$, meaning that $\varphi(x^k, \rho_k) \ge \epsilon/2$.  This contradicts Lemma~\ref{lem.varphi}, which shows that $\varphi(x^k, \rho_k)$ is monotonically decreasing.  Therefore, the set of limit points of $\{x^k\}$ must be all feasible or all infeasible. 
	\end{proof}

	\textcolor{blue}{
		In the following results, we look further at the case  that $\rho_k \to 0$.  
		We show that in this case any limit point of the algorithm is an infeasible stationary point 
		or a feasible point satisfying the \emph{approximate KKT} (AKKT)  \cite{akkt1,akkt2}, which is defined as follows. 
	}
	
	\textcolor{blue}{
		\begin{definition}  Assume that $v(x^*)=0$. We say that $x^*$ satisfies AKKT if there exist sequences
			$\{x^k\} \subset \mathbb{R}^n$ ($\{x^k\}$ is called an AKKT sequence), 
			$\{\mu^k\} \subset \mathbb{R}^{|\Ecal\cup\Ical|}$ such that $\lim_{k\to\infty} x^k = x^*$ and 
			\begin{align}
			\lim_{k\to\infty} \nabla f(x^k) + \sum_{i\in\Ecal \cup \Ical} \mu_i^k \nabla c_i(x^k) = 0, \label{eq.akkt1} \\
			\lim_{k\to\infty} \min\{ \mu_i^k, -c_i(x^k)\} = 0, \quad i\in\Ical. \label{eq.akkt2}
			\end{align}
		\end{definition}
	}
	%
	\textcolor{blue}{
		Defining  
		\[  \Acal(x)  = \{i\in \Ical : c_i(x) = 0\} \ \ \text{and}\ \ 
		\Ncal(x)  = \{i\in\Ical : c_i(x) < 0\},
		\] 
		our main result for the case $\rho_k \to 0$ is summarized below. 
	}
	
	\textcolor{blue}{
		\begin{theorem} 
			Suppose Assumption~\ref{ass.subproblem} and \ref{global} hold and $\rho_k \to 0$. Let $x^*$ be a limit point of $\{x^k\}$ with $v(x^*)= 0$. 
			Then $x^*$  satisfies AKKT. 
		\end{theorem}
	}

	\begin{proof}
		\textcolor{blue}{
			Consider a subsequence $\{x^k\}_{\mathcal{K}} \to x^*$.  By  Lemma~\ref{lem.subproblem},
			the subproblem at $x^*$  have $d^*(0,x^*) = 0$.  Therefore,  for sufficiently large $k\in \mathcal{K}$, 
			we know the $k$th subproblem must have sufficiently small optimal solution such that 
			$\|d^*(\rho_k,x^k) \| < \delta$, meaning the trust region constraint is inactive. Notice that the subproblem may not have a unique optimal solution and it is possible 
			that it has optimal solution on the trust region boundary.
			Therefore, for every sufficiently large  $k \in\mathcal{K}$,   there exists $\bar \lambda^k\in \mathbb{R}^{|\Ecal\cup\Ical}|$ such that 
			\begin{align}  \rho_k \nabla f(x^k) + \sum_{i\in\Ecal\cup\Ical} \bar \lambda^k_i \nabla c_i(x^k) = &\ 0 \label{sub.akkt.1}\\   
			\min\{ \bar \lambda^k_i, - (c_i(x^k)+\langle \nabla c_i(x^k), d^k\rangle )\} = &\ 0, \quad i\in\Ical.\label{sub.akkt.2}
			\end{align}
			Letting $\mu^k : = \bar\lambda^k/\rho^k$, it follows from \eqref{sub.akkt.1} that 
			\[ \nabla f(x^k) +  \sum_{i\in\Ecal\cup\Ical} \mu^k_i \nabla c_i(x^k)  = \nabla f(x^k) +  \sum_{i\in\Ecal\cup\Ical} \frac{\bar\lambda_i^k}{\rho^k} \nabla c_i(x^k) = 0, \]
			meaning \eqref{eq.akkt1} is satisfied. 
		}

		\textcolor{blue}{
			Now we verify \eqref{eq.akkt2}.  
			For each $i\in\Ncal(x^*)$ and sufficiently  large $k\in \mathcal{K}$, we have 
			$c_i(x^k) < c_i(x^*)/2 < 0$ and $\langle \nabla c_i(x^k), d^k\rangle < -c_i(x^*)/4$, implying 
			$c_i(x^k)+\langle \nabla c_i(x^k), d^k\rangle  < 0.$  It follows from \eqref{sub.akkt.2} 
			that $\bar\lambda_i^k = 0$. Hence for sufficiently large $k\in \mathcal{K}$ and $i\in \Ncal(x^*)$, we have  $\mu_i^k =0$ which proves \eqref{eq.akkt2}. 
			On the other hand, for each $i \in \Acal(x^*)$, since $\lim_{k\to \infty, k\in \mathcal{K}} c_i(x^*) = 0$,  
			it holds naturally that 
			$\lim_{k\to \infty, k\in \mathcal{K}} \min \{ \lambda_i^k/\rho^k , -c_i(x^*) \}= 0.$
			Overall, we have shown that $x^*$ satisfies   AKKT, completing the proof. 
		}
	\end{proof}

	\textcolor{blue}{
\emph{Strict Constraint Qualifications} (SCQ) together with AKKT guarantees the KKT conditions are satisfied \cite{akkt1, akkt2}, i.e., 
		\[ \text{AKKT + SCQ} \implies \text{KKT}. \]
		It is shown that many well-known CQs are SCQs, such as the Mangasarian-Fromovitz CQ 
		and linear independence of the gradients of active constraints (LICQ), and the weakest possible SCQ known in the literature is so-called the Cone-continuity Property (CCP) \cite{akkt1}, which is defined as following.   
	}
	
	\textcolor{blue}{
		\begin{definition}[Cone-Continuity Property (CCP)]\cite{akkt1}
			We say that $x \in \Omega:=\{x\in \mathbb{R}^n \mid c_i(x) = 0, i\in \Ecal \ \text{ and }\ c_i(x)\le 0, i\in \Ical\}$
			satisfies CCP if the set-valued mapping (multifunction) 
			$\mathbb{R}^n \ni x \rightrightarrows K(x):= \{ \sum_{i\in\Ecal} \mu_i \nabla c_i(x) + \sum_{ i\in\Acal(x)} \mu_i \nabla c_i(x)
			\ \text{ with } \ \mu_i \ge 0, i\in \Acal(x^*)\}$, is outer semicontinuous at $x^*$, that is,
			\[  \limsup_{x\to x^*} K(x) \subset K(x^*).\]
		\end{definition} 
	}

	\textcolor{blue}{
		In the following corollary, we summarize the results of all of our theorems.
		\begin{corollary}\label{main.corollary} Suppose Assumption~\ref{ass.subproblem} and \ref{global} hold. Then, exactly one of the following occurs
			\begin{enumerate}
				\item[(i)]  $\rho_k \to \rho_*$ for some constant $\rho_*>0$ and each limit point of $\{x^k\}$ either corresponds 
				to a KKT point or an infeasible stationary point for problem~\eqref{prob.nlp}.
				\item[(ii)] $\rho_k \to 0$ and all limit points of $\{x^k\}$ are infeasible stationary points for \eqref{prob.nlp}.
				\item[(iii)]  $\rho_k \to 0$, all limit points of $\{x^k\}$ are feasible for \eqref{prob.nlp}, and are 
				either KKT points or points where the CCP  fails. 
			\end{enumerate}
		\end{corollary}
	}

	It should be noticed that the results above discuss cases where $\{x^k\}$  has limit points.  If this is not the case, meaning $\{x^k\}$ is unbounded,  we can still show the optimality residuals (KKT error) converge for both the penalty problem and feasibility problem.  In fact, in all the cases we can show the worst-case complexity of optimality residuals, which is discussed in the next subsection.

	\subsection{Worst-case complexity for KKT residuals}
	
	In this subsection, we aim to show the worst-case complexity of the KKT residuals  for both the penalty problem and the feasibility problem, which are denoted as 
	\[
	\begin{aligned}
	E_{opt}(x, \lambda,\rho) & = \|\rho \nabla f(x) + \sum_{i\in\Ecal\cup\Ical} \lambda_i \nabla c_i(x) \|_*\\
	\text{and } \  E_{fea}(x, \nu) & = \|  \sum_{i\in\Ecal\cup\Ical} \nu_i \nabla c_i(x) \|_*,
	\end{aligned}
	\]
	The subproblem always chooses dual feasible variables $\lambda^k$ and $\nu^k$ satisfying \eqref{kkt dual feasible opt}. 
	Therefore, we   verify the satisfaction of complementarity \eqref{kkt comp opt} by defining the complementary residual as
	\[ E_{c}(x,\lambda) = \sum_{c_i>0}(1- \lambda_i)v_i(c_i(x))+\sum_{i\in\Ecal,c_i<0}(1+\lambda_i)v_i(c_i(x))+\sum_{i\in\Ical,c_i<0}\lambda_i |c_i(x)|.  \] 
	If $E_{opt}(x^k, \lambda^k,\rho_k) = 0$, $E_{c}(x^k,\lambda^k)=0$ and $v(x^k) = 0$ with $\rho_k>0$, we know $x^k$ is stationary for \eqref{prob.nlp}.  If $E_{fea}(x, \nu)$,  $E_{c}(x^k,\nu^k)=0$ and $v(x^k)>0$, then $x^k$ is an infeasible stationary point.

	Obviously, the KKT residual complexities depend on many factors especially the subproblem tolerance $\{\gamma_k\}$, since they 
	represent how accurately the subproblems are solved.  We make the following assumption about $\{\gamma_k\}$. 
	
	\begin{assumption}\label{gamma sum} The subproblem tolerance $\{\gamma_k\}$ are selected such that 
		$\gamma_k \le \eta k^{-\zeta/2}$ with constant $\eta>0$ and $\zeta\ge1$.
	\end{assumption}
	
	The parameters $\eta$ and $\zeta$ control accuracy of the subproblem solution. Larger $\zeta$ or small $\eta$ means more accurate subproblem solution is needed.

	The following lemma establishes  the relationship between the KKT residual and complementarity residual for feasibility and optimality problems.
	\begin{lemma}\label{kkt error and red}
		Under Assumption~\ref{ass.subproblem}, 
		\ref{global} and \ref{gamma sum}, it holds that, for all $k \in \mathbb{N}$,
		\begin{align}
		E_{opt}(x^k, \lambda^k, \rho^k) & \le \frac{1}{\delta\beta_\phi}\Delta l(d^k; \rho_k, x^k)+ \frac{1-\beta_\phi}{\delta\beta_\phi} \gamma_k ,\label{e.opt.delta1}\\
		E_{c}(x^k, \lambda^k) & \le \frac{1}{  \beta_\phi}\Delta l(d^k; \rho_k, x^k)+ \frac{1-\beta_\phi}{  \beta_\phi} \gamma_k ,\label{e.opt.delta2}\\
		E_{fea}(x^k, \nu^k) & \le \frac{1}{\delta\beta_v}\Delta l(d^k; 0, x^k)+ \frac{1-\beta_v}{\delta\beta_v} \gamma_k,\label{e.opt.delta3}\\
		\text{and }\  E_{c}(x^k, \nu^k) & \le \frac{1}{  \beta_\phi}\Delta l(d^k; 0, x^k)+ \frac{1-\beta_v}{  \beta_v} \gamma_k .\label{e.opt.delta4}
		\end{align}
	\end{lemma}

	\begin{proof}
		For  dual feasible $\lambda^k  $  
		\begin{equation}\label{vl.comp} v(x^k) - \sum_{i\in\Ecal\cup\Ical} \lambda^k_i c_i(x^k) 
		=    E_{c}(x^k,\lambda^k)
		\end{equation}
		which follows  from the fact that (where we temporarily use $c_i^k=c_i(x^k)$ due to space limit)
		\[\begin{cases}
		v_i(c_i^k) - \lambda_i^k c_i^k= v_i(c_i^k) - \lambda_i^k v_i(c_i^k) = (1-\lambda_i^k)v_i(c_i^k) & \text{if } c_i^k > 0,\\
		v_i(c_i^k) - \lambda_i^k c_i^k= v_i(c_i^k) + \lambda_i^kv_i(c_i^k) =  (1+\lambda_i^k)v_i(c_i^k) & \text{if } c_i^k < 0,  i\in\Ecal\\
		v_i(c_i^k) - \lambda_i^k c_i^k = 0 - \lambda_i^k  c_i^k  =   \lambda_i^k  |c_i^k|, &\text{if }  c_i^k < 0, i\in\Ical .
		\end{cases} \]
		Then
		\[\begin{aligned}
		l(0; \rho_k, x^k) - p(\lambda^k; \rho_k, x^k) & = \delta E_{opt}(x^k, \lambda^k, \rho_k) + v(x^k) - \sum_{i\in\Ecal\cup\Ical} \lambda^k_i c_i(x^k)\\
		& = \delta E_{opt}(x^k, \lambda^k, \rho_k) +  E_{c}(x^k,\lambda^k). 
		\end{aligned}  \]
		On the other hand, it holds that 
		\[ l(0; \rho_k, x^k) - p(\lambda^k;  \rho_k, x^k) \le \frac{1}{\beta_\phi} \Delta l(d^k; \rho_k, x^k) + \frac{1-\beta_\phi}{\beta_\phi} \gamma_k\]
		from \eqref{red.penalty}. 
		Combining the above yields \eqref{e.opt.delta1} and \eqref{e.opt.delta2}.   The same argument applied to 
		$l(0; 0, x^k) - p(\lambda^k;  0, x^k)$ and \eqref{red.fea}  proves  \eqref{e.opt.delta3} and \eqref{e.opt.delta4}. 
	\end{proof}
	
	As the sequence $\{\varphi(x^k,\rho_k)\}$ has been shown in Lemma~\ref{lem.varphi} to be monotonically decreasing, we can denote the initial penalty function value $\varphi^0:=\varphi(x^0, \rho_0)$ and the limit 
	$\varphi^* := \lim\limits_{k\to\infty} \varphi(x^k,\rho_k)$ and derive the following complexity results for model reductions. 
	\begin{lemma}\label{lem.delta.complexity}
		Under Assumption~\ref{ass.subproblem}, 
		\ref{global} and \ref{gamma sum},  for any $\epsilon>0$, the following statements hold true
		\begin{enumerate}
			\item[(i)]  It needs at most 
			\[ \frac{\delta\kappa_0\kappa_1( \varphi^0 - \varphi^* )}{\theta_\alpha \beta_\alpha(1-\beta_\alpha)} \frac{1}{\epsilon^2}   \] 
			iterations to reach $\inf_{i = 0}^k \Delta l(d^i; \rho_i, x^i) \le \epsilon$.  
			\item[(ii)]  It needs at most 
			\[ \max\Big\{\frac{4\delta\kappa_0\kappa_1(\varphi^0 - \varphi^* )}{\theta_\alpha  \beta_l^{2} \beta_\alpha(1-\beta_\alpha)} \frac{1}{\epsilon^2}, \left[\frac{2\eta(1-\beta_l)}{\beta_l  \epsilon} \right]^{\frac{2}{\zeta}}\Big\} \] 
			iterations to reach $\inf_{i = 0}^k \Delta l(d^i; 0, x^i) \le \epsilon$.  
		\end{enumerate}
	\end{lemma}

	\begin{proof}  For Part $(i)$, from  Lemma~\ref{lem.varphi},  summing up both sides of \eqref{var.red.rho} from 0 to $k$ gives
		\begin{align}
		\sum_{t=0}^k [\Delta l(d^t; \rho_t, x^t)]^2  \le  \frac{\delta\kappa_0\kappa_1}{\theta_\alpha(1-\beta_\alpha)\beta_\alpha} [\varphi(x^0,\rho_0) - \varphi(x^{k+1}, \rho_{k+1})  ] \ \forall k \in \mathbb{N}.\label{inf.red.rho}
		\end{align}
		Therefore, 
		\begin{align*}
		\inf_{i = 0}^k [\Delta l(d^i; \rho_i, x^i)]^2  \le  \frac{\delta\kappa_0\kappa_1}{ k  \theta_\alpha(1-\beta_\alpha)\beta_\alpha} [\varphi(x^0,\rho_0) - \varphi^* ], 
		\end{align*}
		completing the proof of $(i)$. 
		
		It follows from  \eqref{dust.after} that  
		\[ \Delta l(d^k; 0, x^k)  \le \tfrac{1}{\beta_l }  \Delta l(d^k; \rho_k, x^k) + \tfrac{1-\beta_l}{\beta_l}\gamma_k  .\]
		Part $(i)$ and Assumption~\ref{gamma sum} implies if 
		\[ k \ge \max\Big\{\frac{4\delta\kappa_0\kappa_1[\varphi(x^0,\rho_0) - \varphi^* ]}{\theta_\alpha  \beta_l^{2} \beta_\alpha(1-\beta_\alpha)} \frac{1}{\epsilon^2}, \left[\frac{2\eta(1-\beta_l)}{\beta_l  \epsilon} \right]^{\frac{2}{\zeta}}\Big\},\]
		then
		\[
		\inf_{i = 0}^k \tfrac{1}{\beta_l }  \Delta l(d^i; \rho_i, x^i)  \le  \epsilon/2 \quad \text{ and }\quad 
		\tfrac{1-\beta_l}{\beta_l} \gamma_k \le\epsilon/2,
		\] 
		completing the proof.
	\end{proof}

	Lemma~\ref{kkt error and red} and \ref{lem.delta.complexity} immediately lead to our main results.

	\begin{theorem}\label{oracle}
		Under Assumption~\ref{ass.subproblem}, 
		\ref{global} and \ref{gamma sum},     given $\epsilon>0$, the following statements hold true. 
		\begin{enumerate}
			\item[(i)]  It requires at most 
			\[ \max\Big\{\frac{ 4\kappa_0\kappa_1 (\varphi^0 - \varphi^* )}{\delta \beta_\phi^2   \theta_\alpha \beta_\alpha (1-\beta_\alpha)} \frac{1}{\epsilon^2},[\frac{2\eta(1-\beta_\phi)}{\delta\beta_\phi } \frac{1}{\epsilon} ]^{\frac{2}{\zeta}}\Big\} \] 
			iterations to reach $\inf_{i=0}^k E_{opt}(x^i, \lambda^i,\rho_i)   \le \epsilon$.
			\item[(ii)] It requires at most
			\[ \max\Big\{\frac{ 4 \delta \kappa_0\kappa_1  (\varphi^0 - \varphi^* )}{ \beta_\phi^2   \theta_\alpha \beta_\alpha (1-\beta_\alpha)} \frac{1}{\epsilon^2}, [\frac{2\eta(1-\beta_\phi)}{ \beta_\phi } \frac{1}{\epsilon}]^{\frac{2}{\zeta}}\Big\}  \] 
			iterations to reach $\inf_{i=0}^k E_{c}(x^i, \lambda^i)   \le \epsilon$.
			\item[(iii)]  It requires at most 
			\[ \max\Big\{\frac{16 \kappa_0\kappa_1 (\varphi^0 - \varphi^* )}{\delta \theta_\alpha \beta_v^2  \beta_l \beta_\alpha(1-\beta_\alpha)} \frac{1}{\epsilon^2}, \left[ \frac{2\eta(1-\beta_l)}{\delta \beta_v \beta_l \epsilon}\right]^{\frac{2}{\zeta}}, \left[\frac{2\eta(1-\beta_v)}{\delta\beta_v \epsilon} \right] ^{\frac{2}{\zeta}} \Big\}\] 
			iterations to reach $\inf_{i=0}^k E_{fea}(x^i, \nu^i)   \le \epsilon$.
			\item[(iv)]  It requires at most 
			\[ \max\Big\{\frac{16 \kappa_0\kappa_1 (\varphi^0 - \varphi^* )}{ \theta_\alpha \beta_\phi^2  \beta_l \beta_\alpha(1-\beta_\alpha)} \frac{1}{\epsilon^2}, \left[ \frac{2\eta(1-\beta_l)}{ \beta_\phi \beta_l \epsilon}\right]^{\frac{2}{\zeta}},  \left[\frac{2\eta(1-\beta_v)}{ \beta_v\epsilon} \right] ^{\frac{2}{\zeta}} \Big\}\] 
			iterations to reach $\inf_{i=0}^k E_{c}(x^i, \nu^i)  \le \epsilon$.
			
		\end{enumerate}
	\end{theorem}
	

	\begin{proof} Part $(i)$ can be derived by requiring 
		$\frac{1}{\delta\beta_\phi}\inf_{i=0}^k\Delta l(d^i; \rho_i, x^i)\le \epsilon/2$ and $\frac{1-\beta_\phi}{\delta\beta_\phi} \gamma_k\le \epsilon/2$
		and then  combining     \eqref{e.opt.delta1}  and Lemma~\ref{lem.delta.complexity}$(i)$. 
		
		Part $(ii)$ can be derived by requiring 
		$\frac{1}{ \beta_\phi}\inf_{i=0}^k\Delta l(d^i; \rho_i, x^i)\le \epsilon/2$ and $\frac{1-\beta_\phi}{ \beta_\phi} \gamma_k\le \epsilon/2$
		and then  combining     \eqref{e.opt.delta2}  and Lemma~\ref{lem.delta.complexity}$(i)$. 
		
		Part $(iii)$ is from   \eqref{e.opt.delta3},  Lemma~\ref{lem.delta.complexity}$(ii)$ by  requiring 
		$\frac{1}{\delta \beta_v}\inf_{i=0}^k\Delta l(d^i; 0, x^i)\le \epsilon/2$ and $\frac{1-\beta_v}{\delta \beta_v} \gamma_k\le \epsilon/2$.

		Part $(iv)$ is from   \eqref{e.opt.delta4},  Lemma~\ref{lem.delta.complexity}$(ii)$ by requiring 
		$\frac{1}{ \beta_\phi}\inf_{i=0}^k\Delta l(d^i; 0, x^i)\le \epsilon/2$ and $\frac{1-\beta_v}{  \beta_v} \gamma_k\le \epsilon/2$.
		
	\end{proof}
	
	\subsection{Local complexity of constraint violation} 
	
	We have summarized  the (global) complexity of stationarity and complementarity for both feasible and infeasible cases   in Theorem~\ref{oracle}, and the dual feasibility is maintained all the time during the iteration of the algorithm.  Therefore, we still need to analyze the complexity of primal feasibility when the iterates converge to an optimal solution.  Notice that this is not a concern  in the infeasible case, since Theorem~\ref{oracle} 
	is sufficient for the complexity of KKT residuals of the feasibility problems.  Therefore, in this section, we assume that  $\{x^k\}$ only has feasible limit points.  
	
	The analysis of the behavior $v(x)$ may rely on the monotonic behavior of the penalty function.  However, from Corollary~\ref{main.corollary}, one cannot expect   $v(x)$ decreases steadily over the iterations.  In early iterations, it could happen that the constraint violation continues deteriorating while the objective is improving.  Instead we should focus on the local behavior of $v(x)$ around a limit point $x^*$. 
	Our analysis for $v(x)$ applies to the case that   $\{x^k\}$ converges to a feasible $x^*$ where  strict complementarity is satisfied.
	
	We summarize the local complexity of constraint violation $v(x)$ of $\{x^k\}$ in the following theorem.  
	
	\begin{theorem}\label{thm.vlocal}
		Under Assumption~\ref{ass.subproblem}, 
		\ref{global} and \ref{gamma sum}, suppose  that  $ \{(x^k, \lambda^k) \} \to (x^*,\lambda^*) $ with $v(x^{*}) = 0$  and $- e < \lambda^*_{\Ecal} < e, 0 < \lambda^*_{\Acal} < e$.  
		Then for any $0< \tau < 1- \| \lambda^*\|_\infty$, there exists $\bar k \in \mathbb{N}$ such that the following statements hold true: 
		\begin{enumerate}
			\item[(i)]  
			$ v(x_k) \le E_c(x^k; \lambda^k ) / \tau$ for any $k>\bar k$.
			\item[(ii)] It requires at most 
			\[ \max\Big\{ \frac{ 4 \delta \kappa_0\kappa_1 [\varphi(x^{\bar{k}},\rho_{\bar{k}}) - \varphi^* ]}{ \tau \beta_\phi^2   \theta_\alpha \beta_\alpha (1-\beta_\alpha)} \frac{1}{\epsilon^2}, \frac{1}{\tau}[\frac{2\eta(1-\beta_\phi)}{ \beta_\phi } \frac{1}{\epsilon}]^{\frac{2}{\zeta}} \Big\} \] 
			additional iterations  to reach $\inf_{i = \bar{k}}^kv(x^i)\le \epsilon$ for given $\epsilon > 0$. 
		\end{enumerate} 
	\end{theorem}

	\begin{proof}  We first prove Part $(i)$.  
		Given $0< \tau < 1- \| \lambda^*\|_\infty$, there exists $\bar{k} \in \mathbb{N}$ such that for all $k \ge \bar{k}$, the following holds
		\begin{align*}
		-1+\tau \le \lambda^k_i \le 1-\tau,  &\quad   i\in\Ecal\\
		0< \lambda^k_i \le 1-\tau,                &\quad  i\in \Acal(x^*)\\
		0\le \lambda^k_i \le \tau,\  \  \   \  \                 &\quad i\in\Ncal(x^*).\\
		c_i(x^k) \le 0 ,\  \  \   \  \                 &\quad i\in\Ncal(x^*).
		\end{align*}
		Therefore,  
		\begin{align*}
		v_i(c_i^k) - \lambda_i^k c_i^k & = |c_i^k| - \lambda_i^k c_i^k \ge   |c_i^k| - (1-\tau) |c_i^k|  \ge   \tau |c_i^k| ,  &&   i\in\Ecal  \\
		v_i(c_i^k) - \lambda_i^k c_i^k & = (c_i^k)_+ - \lambda_i^k c_i^k \ge (c_i^k)_+ - (1-\tau) (c_i^k)_+  \ge  \tau (c_i^k)_+ ,  &&   i\in\Acal(x^*)  \\
		v_i(c_i^k) - \lambda_i^k c_i^k & = (c_i^k)_+ - \lambda_i^k c_i^k  = 0 - \lambda_i^k c_i^k \ge 0,  &&   i\in\Ncal(x^*) .
		\end{align*}
		Hence 
		\[ v(x^k)-  \sum_{i\in\Ecal\cup\Ical}   \lambda_i^k c_i^k  = \sum_{i\in\Ecal\cup\Ical} (v_i(c_i^k) - \lambda_i^k c_i^k)  \ge \tau v(x^k). \]
		This, combined with \eqref{vl.comp}, yields 
		$E_c(x^k; \lambda^k ) \ge \tau v(x^k)$, 
		completing the proof of Part $(i)$.
		
		Part $(ii)$   follows naturally from Theorem~\ref{oracle}$(ii)$ by replacing starting point $x^0$ with $x^{\bar{k}}$.
	\end{proof}

	We emphasize that the local complexity result is derived under quite {\color{blue}mild} assumptions compared with other nonlinear optimization methods. 
	The strictly complementary condition  in Theorem~\ref{thm.vlocal} is commonly used in penalty-SQP methods for analyzing the local convergence rate \cite{ byrd2010infeasibility, burke2014sequential}.    Indeed,   second-order methods (interior point methods or SQP methods)  for   constrained nonlinear optimization generally analyze local convergence by assuming strictly complementary condition, regular condition  and second-order sufficient condition.  These three conditions are also required to hold true in \cite{Fletcher} for analyzing the local behavior of the SLP algorithm. 
	
	The other aspect to notice is about the constant $\zeta$, which controls how fast $\gamma_k$ tends to 0. The complexity results we have derived consists $O(\epsilon^{-2})$ and $O(\epsilon^{-2/\zeta})$.  If we choose $\zeta \ge 1$, meaning $\gamma_k\sim O(k^{-1})$, then 
	overall we have $O(\epsilon^{-2})$ in the complexity results.  On the contrary, if we want to drive $\gamma_k$ to zero slower than $O(k^{-1})$, then we have $O(\epsilon^{-2/\zeta})$ in the complexity results.

	\section{Subproblem algorithms} \label{subal}
	
	In this section, we apply   simplex methods to solve the subproblem by focus on $\ell_\infty$ norm trust region in \eqref{primal sub} . Since   the discussion focus on the subproblem at 
	the $k$th iteration,  we drop $x^k$ and the iteration number $k$ and use the shorthand notation as introduced in \S\ref{sec.update}.
	\textcolor{blue}{It should be noticed that using primal simplex is not required in our proposed method, and our framework can accept any 
		subproblem solver that can generate primal-dual feasible iterates. }

	
	Using $\ell_\infty$ norm trust region in \eqref{primal sub}  results in subproblem
	\begin{equation}\label{sub.primal.lp}
	\baligned
	\min_{(d, r, s, t)} &\   \langle \rho g, d\rangle + \langle e, r+s\rangle + \langle e, t\rangle \ \\
	\st &\  \langle a_i, d\rangle + b_i = r_i-s_i, \ i\in\Ecal, \\   
	&\  \langle a_i, d\rangle + b_i \le t_i, \qquad\  i\in\Ical, \\   
	& -\delta e \le  d \le \delta e, \   (r,s,t) \ge 0.
	\ealigned
	\end{equation}
	by adding auxiliary variables $(r, s, t)$. To see how a primal simplex method could benefit from the structures of \eqref{sub.primal.lp}, 
	we  rewrite the  standard form  of  \eqref{sub.primal.lp}  as 
	\begin{equation}\label{primal.standard}
	\min_{\bar x \in \mathbb{R}^{\bar n}} \   \bar c^T \bar x \quad\quad \st\ \ \bar A \bar x = \bar b,\  \bar x \ge 0
	\end{equation}
	with $\bar n = {4n+2|\Ecal|+2|\Ical|}$ by splitting $d$ into $(d_+, d_-)$ and adding slack variables $(z, u, v)$, 
	where 
	\begin{equation*}
	\bar c = 
	\begin{bmatrix}
	\rho g\\
	-  \rho g\\
	e\\ 
	e \\ 
	- e\\
	0 \\
	0 \\
	0
	\end{bmatrix}, 
	\bar x = 
	\begin{bmatrix} d_+ \\ d_- \\ r \\ t \\ s \\ z \\ u \\ v    
	\end{bmatrix},
	\bar A = 
	\begin{bmatrix}
	a_{\Ecal}^T & -a_{\Ecal}^T  &  - I_{\Ecal} & 0 &  I_{\Ecal} & 0  & 0 & 0 \\
	a_{\Ical}^T  & -a_{\Ical}^T   &  0 & -I_{\Ical} & 0 & I_{\Ical} & 0 & 0 \\
	I_n & - I_n & 0 & 0 & 0 & 0 & I  & 0  \\
	-I_n& I_n & 0 & 0 & 0 & 0  &  0 & I 
	\end{bmatrix}, 
	\bar b 
	= \begin{bmatrix} -b_{\Ecal} \\ -b_{\Ical} \\  \delta e \\  \delta e  \end{bmatrix}.   
	\end{equation*}
	The initial tableau can be set as 
	

	\begin{table}[!htb]
		\begin{minipage}{.4\linewidth}
			\centering
			\caption{Initial simplex tableau}\label{tab.tableau1}
			\renewcommand\arraystretch{1.2}
			\begin{tabular}{cccc}
				basic  &  nonbasic    &    rhs    &  dual      \\ \hline
				\multicolumn{1}{|c|}{$\bar B$}   &    \multicolumn{1}{|c|}{$\bar N$}  &     \multicolumn{1}{|c|}{$ \bar b$}  &    \multicolumn{1}{|c|}{$I_{N}$}     \\   \hline
				\multicolumn{1}{|c|}{$-\bar c_B$}    &   \multicolumn{1}{|c|}{$  - \bar c_{N}$}   &  \multicolumn{1}{|c|}{$0$} &  \multicolumn{1}{|c|}{0}      \\ \hline
			\end{tabular} 
		\end{minipage}%
		$\ \underrightarrow{\text{pivot}}\ $
		\begin{minipage}{.5\linewidth}
			\centering
			\caption{Simplex tableau}\label{tab.tableau2}
			\renewcommand\arraystretch{1.2}
			\begin{tabular}{cccc}
				basic  &  nonbasic    &    rhs    &  dual      \\  \hline
				\multicolumn{1}{|c|}{$I_{N}$}   &    \multicolumn{1}{|c|}{$\bar B^{-1} \bar N$}  &     \multicolumn{1}{|c|}{$\bar B^{-1} \bar b$}  &    \multicolumn{1}{|c|}{$\bar{B}^{-1}$}     \\ \hline  
				\multicolumn{1}{|c|}{0}    &   \multicolumn{1}{|c|}{$\bar c_B \bar B^{-1} \bar N - \bar c_{N}$}   &  \multicolumn{1}{|c|}{$\bar c_B \bar B^{-1} \bar b$} &  \multicolumn{1}{|c|}{$\lambda$}      \\  \hline
			\end{tabular}
		\end{minipage} 
	\end{table}
	
	
	The benefits of using a simplex for solving such a linear optimization subproblem in our proposed method can be summarized as follows. 
	\begin{itemize}
		\item  The linear optimization subproblem~\eqref{sub.primal.lp} is always feasible and bounded due to the presence  of slack variables and the trust region.
		\item There exists a   basic feasible solution for the tableau  
		\[ ( d_+, d_-, r, s, t, z, u, v) = (0, 0, (b_\Ecal)_+, (b_\Ical)_+,  -(b_\Ecal)_-,  -(b_\Ical)_-, \delta e, \delta e),\]
		so that  tableau can always be trivially  initialized. 
		\item \textcolor{blue}{ After each pivot, the multipliers can also be extracted from the tableau ($\lambda=\bar c_B \bar B^{-1}$ in Table~\ref{tab.tableau2}). We can then project 
			those multipliers onto the dual feasible region to ensure    dual feasibility.  }
		\item The quantities $l(d; \rho)$, $p(\lambda; \rho)$ and $v_i( \langle a_i, d\rangle + b_i)$ used for computing ratios $r_\phi$, $r_c$ and $r_v$ can be easily extracted from the tableau. Moreover, $d$ can be extracted easily from the last column of the tableau and $\lambda$ can also be extract from the last row of the tableau.
		\item After reducing $\rho$ during pivoting, it is only needed to change the row of the objective vector in the tableau. With a new $\rho$, the current iterate remains basic feasible, so that the simplex method can continue with a ``warm-start'' basic feasible initial point. 
	\end{itemize}

	\section{ Numerical experiments} \label{exp}
	
	In this section, we test \texttt{FoNCO} on a collections of nonlinear problems.  

	\subsection{Trust region radius updates}
	We fix the trust region radius to simplify the analysis. However, in practice, dynamically adjusting the radius helps to improve the algorithm efficiency. In our implementation, the radius is adjusted as described below.
	Define ratio 
	\[ \sigma_k :=  \frac{\phi(x^k,\rho_k)-\phi(x^k+d^k,\rho_k)}{\Delta l(d^k;  \rho_k, x^k)}.\]
	The trust radius is updated as 
	\[ \delta_{k+1} = \begin{cases}  
	\min(2\delta_k, \delta_{\max}) & \text{   if   } \sigma_k > \bar \sigma\\
	\max(\delta_k/2, \delta_{\min})   &  \text{ if  }  \sigma_k < \underline{\sigma}\\
	\delta_k  & \text{   otherwise,}
	\end{cases}
	\]
	where $0 < \underline{\sigma} < \bar \sigma < 1$ and $\delta_{\max} > \delta_{\min}$ are prescribed parameters.  
	
	We choose $\bar\sigma > \beta_\alpha$ such that the Armijo line search condition holds naturally true if  $\sigma_k>\bar\sigma$.
	In this case, the back-tracking line search is skipped after solving the subproblem.  
	We do not consider repeatedly  rejecting the trust region radius and re-solving the subproblem if $\sigma_k < \underline\sigma$. 
	If the trust region radius is reduced to be smaller than $\delta_{\min}$, we stop further reducing the trust region radius and continue with line search.  
	In either case,  our theoretical  analysis still holds.  
	
	\subsection{Implementation}

	Our code\footnote{\url{https://github.com/DataCorrupted/FoNCO}.} is a prototype Python implementation using package \texttt{NumPy}. 
	Define the relative KKT error as 
	\begin{equation}\label{relative.kkt}
	\epsilon_{kkt} := \frac{\max(E_{opt}(x^k, \lambda^k, \rho_k), E_c(x^k, \lambda^k))}{\max(1, E_{opt}(x^0, \lambda^0, \rho_0), E_c(x^0, \lambda^0))}.
	\end{equation}
	The algorithm is terminated if $\epsilon_{kkt} < 10^{-4}$ and constraint violation $v(x^k) < 10^{-4}$.  Otherwise,  the algorithm is deemed to fail within the maximum number of iterations. Denote $\texttt{Iter}^{ps}$ as the maximum number of iterations for the subproblem solver.  The relaxation parameter $\gamma_k$ is updated as 
	$\gamma_k = \gamma_0   \theta_\gamma^{k-1}$. 
	The parameter values used in our implementation are listed in the following Table~\ref{tab.para}.

	\begin{table}[h]
		\centering
		\caption{Parameters in \texttt{FoNCO}}
		\label{tab.para}
		\begin{tabular}{c|ccccccccc}\hline
			Parameter  & $\rho_0$   &  $\beta_\alpha$  &  $\beta_v$& $\beta_\phi$ & $\beta_l$                  &  $\gamma_0$  &  $\theta_\gamma$   & $\theta_\alpha$    & $\delta_0$    \\ \hline 
			Value   &   $1$    &  $10^{-4}     $  &  $0.3    $& $0.75$       & $0.135$ &     0.01     &  0.7   &   $0.5$  & 1  \\  \hline\hline
			
			Parameter &      $\bar\sigma$     &  $\underline{\sigma}$  &  $\delta_{\min} $  &    $\delta_{\max}$      & $\texttt{Iter}$ & $\texttt{Iter}^{ps}$  &  & & \\  \hline
			Value &           $0.3$  &   $0.75$  &  $64$       &   $10^{-4}$  &        $1024$   &  $100$ & & & \\  \hline
		\end{tabular}
	\end{table}
	We tested our implementation on 126   Hock ÄìSchittkowski problems  \cite{hock1980test}   of   CUTEr \cite{Gould2015} on a ThinkPad T470 with i5-6700U processor. The detailed performance statistics of \texttt{FoNCO} is provided in Table~\ref{table.numerical}, where column name is explained in Table \ref{table.explain}. 
	
	\begin{table}[h]
		\centering
		\caption{Column Explanation}
		\label{table.explain}
		\begin{tabular}{c|l}\hline
			Problem & The name of the problem \\ \hline
			\# iter  & Number of iterations    \\ \hline
			\# pivot & Total number of pivots     \\ \hline
			\# $f$ & Number of function   evaluations   \\ \hline
			$f(x*)$ & Final objective value     \\ \hline
			$v(x^*)$ & Final    constraint violation    \\ \hline
			KKT   & Final relative KKT error defined in \eqref{relative.kkt}   \\ \hline
			$\rho*$ & Final $\rho$   \\ \hline
			Exit & 1(Success) or $-1$(\texttt{Iter} exceeded) \\ \hline
		\end{tabular}
	\end{table}

	We have the following observations from the experiment.
	\begin{itemize}
		\item Our algorithm solves 113 out of these 126 problems, attaining a success rate   $\approx 89.7\%$. We noticed that some problems
		are sensitive to the selection of trust region radius. For examples, problems HS87, HS93 HS101, HS102 and HS103 failed with initial trust region radius $1$. We re-ran those 5 problems with a smaller initial trust region radius $\delta_{0} = 10^{-4}$. All these 5 problems are solved successfully. We believe the robustness of our proposed algorithm could be improved with a more sophisticated trust region radius updating strategy. 
		\item Simplex method employed in our implementation is very efficient. Figure \ref{fig:pivot-per-iteration} shows the histogram of average number of pivots per iteration for 113 successful problems. We can see that for the majority of the cases, pivot per iteration is less than 5.
		\begin{figure}[ht]
			\begin{center}
				\includegraphics[scale = 0.36]{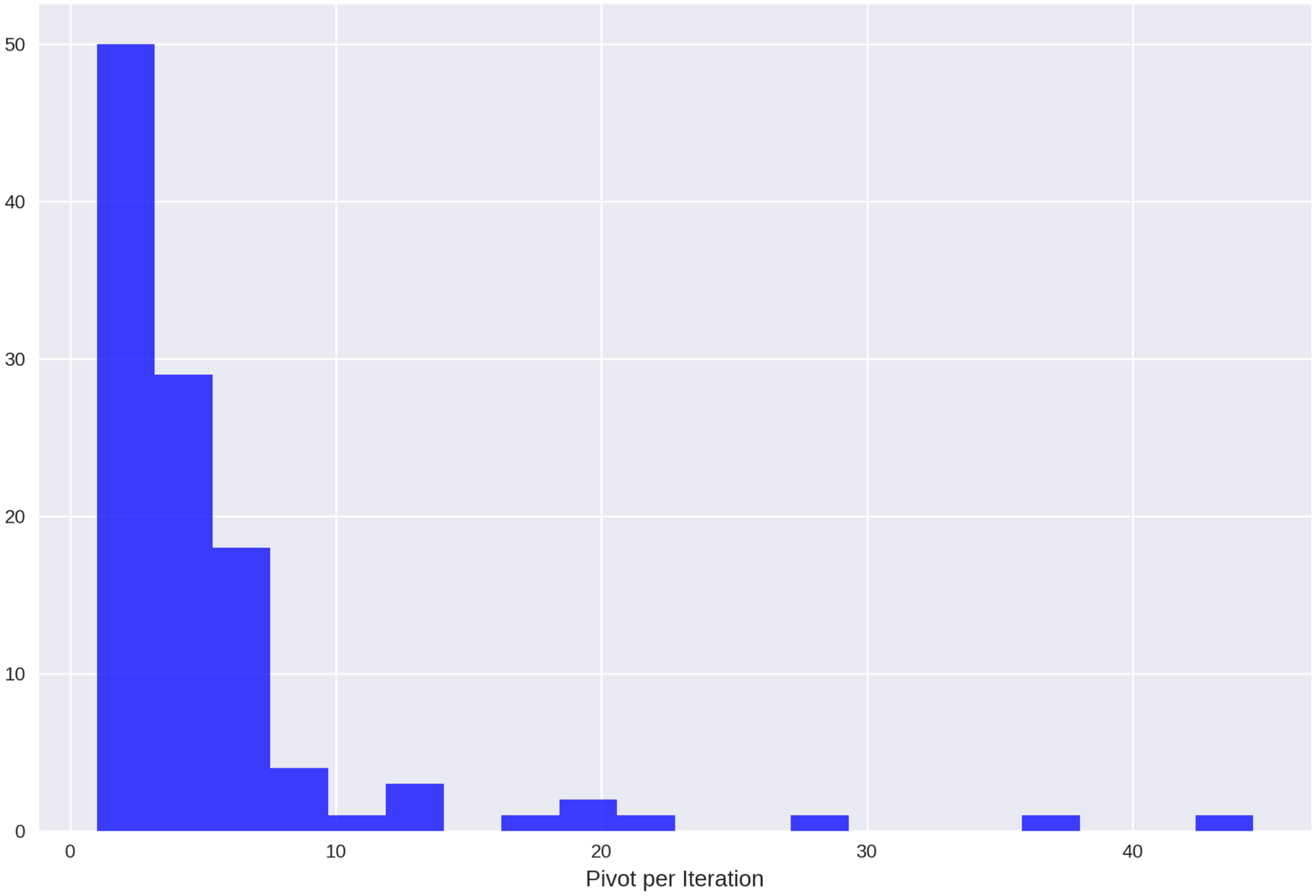}
			\end{center}
			\caption{Pivots per iteration for 113 successful cases.}
			\label{fig:pivot-per-iteration}
		\end{figure}
		\item Compared with second-order methods, SLP may take more outer iterations to compute for a high accuracy solution. However, the lower computational cost of each subproblem might be able to compensate for more outer iterations.
	\end{itemize}

	\textcolor{blue} {
		We also tested our proposed method on a subset of the large-scale CUTEr problems and compared with the exact SLP algorithm. For the exact solver,  we set 
		$\beta_\phi=1$ to enforce the subproblem to be solved accurately.  This    resembles the ``traditional''  exact penalty SLP methods and 
		the classic penalty parameter updating strategy used in \cite{burke1989sequential,han1977globally,han1979exact}  where the 
		penalty parameter is updated after the subproblem is solved. 
		Table \ref{table.comp} shows the test results, where a problem name with subscript $\_ex$ denotes the test results of the 
		exact solver,  $\downarrow$ \% is the improvement percentage on number of pivots by the inexact algorithm compared with the exact version, 
		  \# cons  is the number of constraints, and \# vars  is the number of variables.  We make the following observation from Table \ref{table.comp}, 
			\begin{itemize}
			\item Our proposed method successfully solves the 23 examples, and has 
		fewer number of function evaluations or pivots than the exact SLP method, whereas the exact penalty-SLP method fails at two problems. 
		\item The number of pivots  needed per iteration is still not very large considering the problem sizes. 
		\end{itemize}
		It should be noticed that the results in Table~\ref{table.comp} are preliminary and may be further improved 
		by implementing a powerful large-scale LP solver, e.g., the interior point method. However, this would be out of 
		the scope of this paper's focus.
}

	\section{Conclusion} \label{con}
	In this paper, we have proposed, analyzed, and tested an algorithmic framework of first-order methods for solving nonlinear constrained optimization problems that possesses global convergence guarantees for both feasible and infeasible problems instances.  The worst-case complexity of 
	KKT residuals for feasible and infeasible cases have been studied as well as the local complexity for constraint violation 
	for feasible cases. 
	
	Numerical results demonstrated that the proposed methods work on HS test problems.  We remark, however, the selection of trust region radius and its updating strategy plays a key role in the robustness of the methods. It would be interesting to develop more efficient updating strategies and study how the complexity could be affected by the trust region.

	

\begin{scriptsize}
\begin{longtable} {l||r|r|r||r|r|r|r|r}

\caption{ Numerical Experiment}
\label{table.numerical}\\

\hline
\multicolumn{1}{l}{Problem}    & \multicolumn{1}{c}{\# iter}        & \multicolumn{1}{c}{\# pivot}         &  \multicolumn{1}{c}{  \# $f$}  & 
 \multicolumn{1}{c}{   $f(x^*)$} &  \multicolumn{1}{c}{$v(x^*)$}  &  \multicolumn{1}{c}{ KKT  } &  \multicolumn{1}{c}{ $\rho_*$} &  
 \multicolumn{1}{c}{Exit}
  \\  \hline
\hline
\endfirsthead
\endhead

\multicolumn{9}{c}%
{{  \tablename\ \thetable{} -- continued from previous page}} \\
\hline
\multicolumn{1}{l}{Problem}  & \multicolumn{1}{c}{  \# iter}  & \multicolumn{1}{c}{   \# pivot} &  \multicolumn{1}{c}{  \# $f$}  & 
 \multicolumn{1}{c}{   $f(x^*)$} &  \multicolumn{1}{c}{     $v(x^*)$}  &  \multicolumn{1}{c}{     KKT  } &    \multicolumn{1}{c}{     $\rho_*$} &  \multicolumn{1}{c}{Exit}
  \\  \hline
\hline
\endhead

\hline \multicolumn{6}{c}{{Continued on next page}} \\ \hline
\endfoot

\hline\hline
\endlastfoot

\texttt{      HS1 } &     260 &      504 &       440 &  $7.490597\text{E}{-02}$ &  $0.0\text{E}{+00}$ &  $9.7\text{E}{-05}$ &  $1.6\text{E}{-03}$ & 1 \\
\texttt{     HS10 } &      16 &       25 &        40 & -$1.000002\text{E}{+00}$ &  $4.8\text{E}{-06}$ &  $6.3\text{E}{-06}$ &  $1.0\text{E}{+00}$ & 1 \\
\texttt{    HS100 } &      99 &      545 &       394 &  $6.806299\text{E}{+02}$ &  $9.8\text{E}{-05}$ &  $6.2\text{E}{-05}$ &  $7.0\text{E}{-01}$ & 1 \\
\texttt{ HS100LNP } &     243 &     1653 &      1725 &  $6.806300\text{E}{+02}$ &  $3.2\text{E}{-05}$ &  $9.8\text{E}{-05}$ &  $5.9\text{E}{-01}$ & 1 \\
\texttt{ HS100MOD } &      32 &      129 &        52 &  $6.786796\text{E}{+02}$ &  $1.1\text{E}{-05}$ &  $6.9\text{E}{-05}$ &  $6.2\text{E}{-01}$ & 1 \\
\texttt{    HS101 } &    1025 &     5089 &      8011 &  $3.000007\text{E}{+03}$ &  $7.6\text{E}{-03}$ &  $3.6\text{E}{-13}$ &  $1.0\text{E}{-04}$ & -1 \\
\texttt{    HS102 } &      66 &      311 &        89 &  $1.000775\text{E}{+03}$ &  $4.1\text{E}{-06}$ &  $7.4\text{E}{-05}$ &  $1.0\text{E}{-04}$ & 1 \\
\texttt{    HS103 } &      71 &      382 &       100 &  $7.201554\text{E}{+02}$ &  $6.2\text{E}{-05}$ &  $9.8\text{E}{-05}$ &  $4.3\text{E}{-05}$ & 1 \\
\texttt{    HS104 } &       5 &       27 &         6 &  $4.200000\text{E}{+00}$ &  $8.0\text{E}{-08}$ &  $8.8\text{E}{-16}$ &  $1.0\text{E}{+00}$ & 1 \\
\texttt{    HS105 } &     244 &     1904 &      1045 &  $1.044612\text{E}{+03}$ &  $0.0\text{E}{+00}$ &  $1.4\text{E}{-05}$ &  $2.1\text{E}{-06}$ & 1 \\
\texttt{    HS106 } &    1025 &    15683 &      1117 &  $2.146195\text{E}{+03}$ &  $1.1\text{E}{+00}$ &  $1.1\text{E}{-03}$ &  $7.2\text{E}{-04}$ & -1 \\
\texttt{    HS107 } &    1025 &     7191 &      1078 &  $2.713610\text{E}{+06}$ &  $5.2\text{E}{+00}$ &  $2.3\text{E}{+05}$ &  $2.6\text{E}{-10}$ & -1 \\
\texttt{    HS108 } &      32 &     1198 &       135 & -$8.660254\text{E}{-01}$ &  $2.0\text{E}{-07}$ &  $6.1\text{E}{-05}$ &  $9.0\text{E}{-01}$ & 1 \\
\texttt{    HS109 } &    1025 &     4727 &      1049 &  $0.000000\text{E}{+00}$ &  $5.4\text{E}{+03}$ &  $1.0\text{E}{+00}$ &  $1.0\text{E}{+00}$ & -1 \\
\texttt{     HS11 } &      26 &       60 &        60 & -$8.498486\text{E}{+00}$ &  $7.2\text{E}{-06}$ &  $4.2\text{E}{-05}$ &  $2.6\text{E}{-01}$ & 1 \\
\texttt{    HS110 } &      23 &      170 &        92 & -$4.577848\text{E}{+01}$ &  $0.0\text{E}{+00}$ &  $6.4\text{E}{-05}$ &  $1.0\text{E}{+00}$ & 1 \\
\texttt{    HS111 } &     298 &     3823 &       515 & -$4.776118\text{E}{+01}$ &  $1.1\text{E}{-05}$ &  $8.4\text{E}{-05}$ &  $5.4\text{E}{-02}$ & 1 \\
\texttt{ HS111LNP } &     298 &     3823 &       515 & -$4.776118\text{E}{+01}$ &  $1.1\text{E}{-05}$ &  $8.4\text{E}{-05}$ &  $5.4\text{E}{-02}$ & 1 \\
\texttt{    HS112 } &      54 &      691 &       216 & -$4.776109\text{E}{+01}$ &  $8.1\text{E}{-16}$ &  $6.3\text{E}{-05}$ &  $1.3\text{E}{-03}$ & 1 \\
\texttt{    HS113 } &      34 &      338 &        74 &  $2.430622\text{E}{+01}$ &  $1.5\text{E}{-05}$ &  $8.8\text{E}{-05}$ &  $2.8\text{E}{-01}$ & 1 \\
\texttt{    HS114 } &    1025 &    25291 &      1073 & -$1.636123\text{E}{+03}$ &  $0.0\text{E}{+00}$ &  $1.0\text{E}{+00}$ &  $3.3\text{E}{-47}$ & -1 \\
\texttt{    HS116 } &      10 &       77 &        12 &  $2.500000\text{E}{+02}$ &  $9.0\text{E}{-07}$ &  $3.6\text{E}{-15}$ &  $1.0\text{E}{+00}$ & 1 \\
\texttt{    HS117 } &     352 &     8011 &       466 &  $3.234873\text{E}{+01}$ &  $0.0\text{E}{+00}$ &  $8.2\text{E}{-05}$ &  $3.6\text{E}{-04}$ & 1 \\
\texttt{    HS118 } &      18 &      319 &        19 &  $9.329922\text{E}{+02}$ &  $0.0\text{E}{+00}$ &  $1.9\text{E}{-16}$ &  $7.4\text{E}{-02}$ & 1 \\
\texttt{    HS119 } &      16 &      436 &        20 &  $2.449598\text{E}{+02}$ &  $1.9\text{E}{-15}$ &  $7.3\text{E}{-05}$ &  $1.5\text{E}{-01}$ & 1 \\
\texttt{     HS12 } &      10 &       20 &        15 & -$3.000000\text{E}{+01}$ &  $4.6\text{E}{-10}$ &  $7.7\text{E}{-06}$ &  $1.0\text{E}{+00}$ & 1 \\
\texttt{     HS13 } &      15 &       26 &        16 &  $4.000000\text{E}{+00}$ &  $0.0\text{E}{+00}$ &  $8.5\text{E}{-05}$ &  $1.0\text{E}{-04}$ & 1 \\
\texttt{     HS14 } &       9 &       19 &        16 &  $1.393465\text{E}{+00}$ &  $1.0\text{E}{-14}$ &  $3.0\text{E}{-07}$ &  $4.0\text{E}{-01}$ & 1 \\
\texttt{     HS15 } &      78 &      167 &       132 &  $3.065000\text{E}{+02}$ &  $0.0\text{E}{+00}$ &  $3.2\text{E}{-17}$ &  $1.0\text{E}{-03}$ & 1 \\
\texttt{     HS16 } &      35 &       79 &        77 &  $2.314466\text{E}{+01}$ &  $0.0\text{E}{+00}$ &  $1.9\text{E}{-08}$ &  $2.8\text{E}{-02}$ & 1 \\
\texttt{     HS17 } &      17 &       42 &        38 &  $1.000000\text{E}{+00}$ &  $6.1\text{E}{-21}$ &  $8.4\text{E}{-08}$ &  $4.7\text{E}{-01}$ & 1 \\
\texttt{     HS18 } &      13 &       26 &        34 &  $5.000000\text{E}{+00}$ &  $0.0\text{E}{+00}$ &  $9.4\text{E}{-06}$ &  $1.0\text{E}{+00}$ & 1 \\
\texttt{     HS19 } &     173 &      359 &       189 & -$6.961814\text{E}{+03}$ &  $0.0\text{E}{+00}$ &  $1.9\text{E}{-08}$ &  $9.6\text{E}{-05}$ & 1 \\
\texttt{      HS2 } &      15 &       26 &        54 &  $4.941229\text{E}{+00}$ &  $0.0\text{E}{+00}$ &  $3.1\text{E}{-05}$ &  $9.0\text{E}{-01}$ & 1 \\
\texttt{     HS20 } &      40 &       88 &       314 &  $4.019873\text{E}{+01}$ &  $0.0\text{E}{+00}$ &  $2.3\text{E}{-07}$ &  $8.1\text{E}{-03}$ & 1 \\
\texttt{     HS21 } &       5 &        7 &         6 & -$9.996000\text{E}{+01}$ &  $0.0\text{E}{+00}$ &  $0.0\text{E}{+00}$ &  $1.0\text{E}{+00}$ & 1 \\
\texttt{  HS21MOD } &      15 &       23 &        19 & -$9.596000\text{E}{+01}$ &  $0.0\text{E}{+00}$ &  $0.0\text{E}{+00}$ &  $1.7\text{E}{-01}$ & 1 \\
\texttt{     HS22 } &       5 &       10 &         6 &  $1.000000\text{E}{+00}$ &  $0.0\text{E}{+00}$ &  $3.7\text{E}{-09}$ &  $1.0\text{E}{+00}$ & 1 \\
\texttt{     HS23 } &      18 &       37 &       225 &  $2.000000\text{E}{+00}$ &  $0.0\text{E}{+00}$ &  $1.0\text{E}{-07}$ &  $3.1\text{E}{-01}$ & 1 \\
\texttt{     HS24 } &       3 &        6 &        38 & -$1.000000\text{E}{+00}$ &  $0.0\text{E}{+00}$ &  $1.9\text{E}{-16}$ &  $1.0\text{E}{+00}$ & 1 \\
\texttt{     HS25 } &       1 &        0 &         1 &  $3.283500\text{E}{+01}$ &  $0.0\text{E}{+00}$ &  $1.9\text{E}{-08}$ &  $1.0\text{E}{+00}$ & 1 \\
\texttt{     HS26 } &      86 &      254 &       415 &  $8.505871\text{E}{-06}$ &  $5.4\text{E}{-06}$ &  $9.8\text{E}{-05}$ &  $7.2\text{E}{-01}$ & 1 \\
\texttt{    HS268 } &     249 &     1259 &       827 &  $3.075321\text{E}{+00}$ &  $0.0\text{E}{+00}$ &  $9.6\text{E}{-05}$ &  $1.1\text{E}{-04}$ & 1 \\
\texttt{     HS27 } &      16 &       33 &        22 &  $4.000000\text{E}{-02}$ &  $0.0\text{E}{+00}$ &  $0.0\text{E}{+00}$ &  $3.1\text{E}{-01}$ & 1 \\
\texttt{     HS28 } &       6 &       15 &         9 &  $0.000000\text{E}{+00}$ &  $0.0\text{E}{+00}$ &  $0.0\text{E}{+00}$ &  $9.0\text{E}{-01}$ & 1 \\
\texttt{     HS29 } &     360 &     1230 &      2659 & -$2.262742\text{E}{+01}$ &  $8.2\text{E}{-06}$ &  $7.7\text{E}{-05}$ &  $1.0\text{E}{+00}$ & 1 \\
\texttt{      HS3 } &     518 &     1019 &       519 &  $2.490010\text{E}{-04}$ &  $0.0\text{E}{+00}$ &  $9.9\text{E}{-05}$ &  $1.0\text{E}{-04}$ & 1 \\
\texttt{     HS30 } &       7 &       20 &         8 &  $1.000061\text{E}{+00}$ &  $0.0\text{E}{+00}$ &  $3.0\text{E}{-05}$ &  $1.0\text{E}{+00}$ & 1 \\
\texttt{     HS31 } &     129 &      315 &       599 &  $5.999992\text{E}{+00}$ &  $1.3\text{E}{-06}$ &  $5.2\text{E}{-05}$ &  $1.0\text{E}{-01}$ & 1 \\
\texttt{     HS32 } &      11 &       47 &        17 &  $1.000000\text{E}{+00}$ &  $0.0\text{E}{+00}$ &  $9.2\text{E}{-17}$ &  $2.0\text{E}{-01}$ & 1 \\
\texttt{     HS33 } &     291 &      589 &       292 & -$4.000000\text{E}{+00}$ &  $0.0\text{E}{+00}$ &  $2.6\text{E}{-07}$ &  $5.8\text{E}{-09}$ & 1 \\
\texttt{     HS34 } &      11 &       50 &        12 & -$8.340324\text{E}{-01}$ &  $4.5\text{E}{-10}$ &  $4.2\text{E}{-06}$ &  $1.0\text{E}{+00}$ & 1 \\
\texttt{     HS35 } &      18 &       42 &        87 &  $1.111111\text{E}{-01}$ &  $0.0\text{E}{+00}$ &  $1.3\text{E}{-05}$ &  $6.5\text{E}{-01}$ & 1 \\
\texttt{    HS35I } &      18 &       42 &        87 &  $1.111111\text{E}{-01}$ &  $0.0\text{E}{+00}$ &  $1.3\text{E}{-05}$ &  $6.5\text{E}{-01}$ & 1 \\
\texttt{  HS35MOD } &       2 &        5 &         3 &  $2.500000\text{E}{-01}$ &  $0.0\text{E}{+00}$ &  $0.0\text{E}{+00}$ &  $1.0\text{E}{+00}$ & 1 \\
\texttt{     HS36 } &     376 &     1093 &       377 & -$3.300000\text{E}{+03}$ &  $0.0\text{E}{+00}$ &  $2.8\text{E}{-16}$ &  $2.0\text{E}{-08}$ & 1 \\
\texttt{     HS37 } &     389 &     1136 &       408 & -$3.456000\text{E}{+03}$ &  $0.0\text{E}{+00}$ &  $1.9\text{E}{-05}$ &  $1.5\text{E}{-08}$ & 1 \\
\texttt{     HS38 } &      44 &      139 &        78 &  $1.070841\text{E}{-02}$ &  $0.0\text{E}{+00}$ &  $4.8\text{E}{-05}$ &  $7.8\text{E}{-03}$ & 1 \\
\texttt{     HS39 } &      28 &       83 &        46 & -$1.000044\text{E}{+00}$ &  $4.3\text{E}{-05}$ &  $8.5\text{E}{-06}$ &  $8.1\text{E}{-01}$ & 1 \\
\texttt{   HS3MOD } &      28 &       33 &        29 &  $0.000000\text{E}{+00}$ &  $0.0\text{E}{+00}$ &  $0.0\text{E}{+00}$ &  $5.9\text{E}{-01}$ & 1 \\
\texttt{      HS4 } &       9 &       23 &        16 &  $2.666667\text{E}{+00}$ &  $0.0\text{E}{+00}$ &  $0.0\text{E}{+00}$ &  $1.8\text{E}{-01}$ & 1 \\
\texttt{     HS40 } &      41 &     1826 &       174 & -$2.500000\text{E}{-01}$ &  $1.6\text{E}{-09}$ &  $8.6\text{E}{-05}$ &  $1.5\text{E}{-01}$ & 1 \\
\texttt{     HS41 } &     176 &      636 &       215 &  $1.925926\text{E}{+00}$ &  $0.0\text{E}{+00}$ &  $5.8\text{E}{-05}$ &  $1.0\text{E}{-07}$ & 1 \\
\texttt{     HS42 } &      20 &       62 &        50 &  $1.385786\text{E}{+01}$ &  $1.6\text{E}{-07}$ &  $5.1\text{E}{-05}$ &  $3.1\text{E}{-01}$ & 1 \\
\texttt{     HS43 } &      24 &       93 &        48 & -$4.400000\text{E}{+01}$ &  $1.5\text{E}{-08}$ &  $5.6\text{E}{-05}$ &  $4.5\text{E}{-01}$ & 1 \\
\texttt{     HS44 } &      98 &      398 &        99 & -$1.500000\text{E}{+01}$ &  $0.0\text{E}{+00}$ &  $4.4\text{E}{-16}$ &  $1.0\text{E}{-04}$ & 1 \\
\texttt{  HS44NEW } &      91 &      380 &        92 & -$1.500000\text{E}{+01}$ &  $0.0\text{E}{+00}$ &  $0.0\text{E}{+00}$ &  $1.1\text{E}{-04}$ & 1 \\
\texttt{     HS45 } &       4 &       17 &         5 &  $1.000000\text{E}{+00}$ &  $0.0\text{E}{+00}$ &  $0.0\text{E}{+00}$ &  $1.0\text{E}{+00}$ & 1 \\
\texttt{     HS46 } &      90 &      470 &       248 &  $1.987922\text{E}{-05}$ &  $4.8\text{E}{-06}$ &  $8.8\text{E}{-05}$ &  $1.0\text{E}{+00}$ & 1 \\
\texttt{     HS47 } &      41 &      247 &        80 &  $2.842654\text{E}{-05}$ &  $3.6\text{E}{-05}$ &  $8.2\text{E}{-05}$ &  $8.1\text{E}{-01}$ & 1 \\
\texttt{     HS48 } &       4 &       16 &         6 &  $1.109336\text{E}{-31}$ &  $0.0\text{E}{+00}$ &  $3.7\text{E}{-17}$ &  $1.0\text{E}{+00}$ & 1 \\
\texttt{     HS49 } &      20 &      116 &        27 &  $4.978240\text{E}{-03}$ &  $0.0\text{E}{+00}$ &  $8.1\text{E}{-05}$ &  $2.8\text{E}{-01}$ & 1 \\
\texttt{      HS5 } &       8 &       12 &        39 & -$1.913223\text{E}{+00}$ &  $0.0\text{E}{+00}$ &  $3.8\text{E}{-05}$ &  $1.0\text{E}{+00}$ & 1 \\
\texttt{     HS50 } &      10 &       53 &        17 &  $1.232595\text{E}{-32}$ &  $4.4\text{E}{-16}$ &  $7.0\text{E}{-19}$ &  $1.8\text{E}{-01}$ & 1 \\
\texttt{     HS51 } &      13 &       59 &        23 &  $2.170139\text{E}{-08}$ &  $0.0\text{E}{+00}$ &  $8.3\text{E}{-05}$ &  $7.2\text{E}{-01}$ & 1 \\
\texttt{     HS52 } &      37 &      213 &       109 &  $5.326649\text{E}{+00}$ &  $1.6\text{E}{-16}$ &  $9.0\text{E}{-05}$ &  $5.0\text{E}{-02}$ & 1 \\
\texttt{     HS53 } &      22 &      110 &        45 &  $4.093023\text{E}{+00}$ &  $1.1\text{E}{-16}$ &  $6.4\text{E}{-05}$ &  $1.2\text{E}{-01}$ & 1 \\
\texttt{     HS54 } &       5 &        8 &        37 & -$1.539517\text{E}{-01}$ &  $0.0\text{E}{+00}$ &  $1.4\text{E}{-06}$ &  $1.0\text{E}{+00}$ & 1 \\
\texttt{     HS55 } &       1 &        7 &         2 &  $6.666667\text{E}{+00}$ &  $1.1\text{E}{-16}$ &  $3.7\text{E}{-17}$ &  $1.0\text{E}{+00}$ & 1 \\
\texttt{     HS56 } &       2 &        5 &         3 & -$1.000000\text{E}{+00}$ &  $1.2\text{E}{-15}$ &  $2.7\text{E}{-05}$ &  $1.0\text{E}{-04}$ & 1 \\
\texttt{     HS57 } &       7 &        7 &        61 &  $3.064631\text{E}{-02}$ &  $0.0\text{E}{+00}$ &  $9.6\text{E}{-05}$ &  $1.0\text{E}{+00}$ & 1 \\
\texttt{     HS59 } &       9 &       13 &        10 &  $3.012922\text{E}{+01}$ &  $1.4\text{E}{-07}$ &  $9.5\text{E}{-06}$ &  $1.0\text{E}{+00}$ & 1 \\
\texttt{      HS6 } &     314 &      789 &      1710 &  $1.503772\text{E}{-13}$ &  $7.7\text{E}{-06}$ &  $8.5\text{E}{-05}$ &  $9.0\text{E}{-01}$ & 1 \\
\texttt{     HS60 } &     117 &      397 &       834 &  $3.256821\text{E}{-02}$ &  $2.0\text{E}{-09}$ &  $9.8\text{E}{-05}$ &  $1.0\text{E}{+00}$ & 1 \\
\texttt{     HS61 } &      18 &       60 &        34 & -$1.436462\text{E}{+02}$ &  $5.3\text{E}{-05}$ &  $7.2\text{E}{-05}$ &  $4.9\text{E}{-01}$ & 1 \\
\texttt{     HS62 } &     258 &      745 &      1395 & -$2.627251\text{E}{+04}$ &  $1.6\text{E}{-16}$ &  $8.4\text{E}{-05}$ &  $3.1\text{E}{-05}$ & 1 \\
\texttt{     HS63 } &      16 &       54 &        46 &  $9.617152\text{E}{+02}$ &  $2.7\text{E}{-06}$ &  $1.7\text{E}{-05}$ &  $5.5\text{E}{-01}$ & 1 \\
\texttt{     HS64 } &      43 &      101 &        51 &  $6.299779\text{E}{+03}$ &  $3.4\text{E}{-05}$ &  $1.8\text{E}{-06}$ &  $4.9\text{E}{-02}$ & 1 \\
\texttt{     HS65 } &      23 &       59 &        30 &  $9.535289\text{E}{-01}$ &  $1.7\text{E}{-10}$ &  $4.3\text{E}{-05}$ &  $1.0\text{E}{+00}$ & 1 \\
\texttt{     HS66 } &      23 &       61 &        84 &  $5.181632\text{E}{-01}$ &  $2.4\text{E}{-07}$ &  $7.0\text{E}{-05}$ &  $5.3\text{E}{-01}$ & 1 \\
\texttt{     HS67 } &    1025 &     3071 &      1034 & -$9.162074\text{E}{+02}$ &  $0.0\text{E}{+00}$ &  $7.3\text{E}{-02}$ &  $1.5\text{E}{-04}$ & -1 \\
\texttt{     HS68 } &      38 &      138 &       389 &  $2.400000\text{E}{-05}$ &  $0.0\text{E}{+00}$ &  $4.2\text{E}{-14}$ &  $1.7\text{E}{-03}$ & 1 \\
\texttt{     HS69 } &      62 &      281 &       505 & -$9.280357\text{E}{+02}$ &  $1.1\text{E}{-10}$ &  $4.4\text{E}{-05}$ &  $2.8\text{E}{-08}$ & 1 \\
\texttt{      HS7 } &      17 &       32 &        51 & -$1.732051\text{E}{+00}$ &  $3.0\text{E}{-09}$ &  $4.7\text{E}{-05}$ &  $9.0\text{E}{-01}$ & 1 \\
\texttt{     HS70 } &    1025 &     4056 &      2799 &  $1.866660\text{E}{-02}$ &  $0.0\text{E}{+00}$ &  $1.7\text{E}{-02}$ &  $9.6\text{E}{-03}$ & -1 \\
\texttt{     HS71 } &      28 &      529 &       237 &  $1.701402\text{E}{+01}$ &  $3.9\text{E}{-08}$ &  $9.1\text{E}{-05}$ &  $3.9\text{E}{-01}$ & 1 \\
\texttt{     HS72 } &      99 &      333 &       104 &  $7.276793\text{E}{+02}$ &  $4.4\text{E}{-09}$ &  $7.0\text{E}{-05}$ &  $1.7\text{E}{-05}$ & 1 \\
\texttt{     HS73 } &      31 &      272 &        33 &  $2.989438\text{E}{+01}$ &  $1.5\text{E}{-07}$ &  $2.3\text{E}{-06}$ &  $3.7\text{E}{-02}$ & 1 \\
\texttt{     HS74 } &      51 &      200 &        64 &  $5.126498\text{E}{+03}$ &  $1.4\text{E}{-06}$ &  $4.6\text{E}{-05}$ &  $1.3\text{E}{-01}$ & 1 \\
\texttt{     HS75 } &    1025 &     6031 &      1042 &  $5.127004\text{E}{+03}$ &  $3.1\text{E}{-02}$ &  $4.9\text{E}{-05}$ &  $2.4\text{E}{-03}$ & -1 \\
\texttt{     HS76 } &      23 &       97 &        68 & -$4.681818\text{E}{+00}$ &  $0.0\text{E}{+00}$ &  $2.6\text{E}{-05}$ &  $3.9\text{E}{-01}$ & 1 \\
\texttt{    HS76I } &      23 &       97 &        68 & -$4.681818\text{E}{+00}$ &  $0.0\text{E}{+00}$ &  $2.6\text{E}{-05}$ &  $3.9\text{E}{-01}$ & 1 \\
\texttt{     HS77 } &      49 &      245 &       141 &  $2.415043\text{E}{-01}$ &  $1.8\text{E}{-05}$ &  $7.2\text{E}{-05}$ &  $1.0\text{E}{+00}$ & 1 \\
\texttt{     HS78 } &      19 &      111 &        73 & -$2.919700\text{E}{+00}$ &  $2.7\text{E}{-07}$ &  $4.7\text{E}{-05}$ &  $8.1\text{E}{-01}$ & 1 \\
\texttt{     HS79 } &     226 &     1653 &      1181 &  $7.877686\text{E}{-02}$ &  $2.9\text{E}{-06}$ &  $8.7\text{E}{-05}$ &  $1.0\text{E}{+00}$ & 1 \\
\texttt{      HS8 } &       5 &       10 &         6 & -$1.000000\text{E}{+00}$ &  $6.4\text{E}{-07}$ &  $0.0\text{E}{+00}$ &  $1.0\text{E}{+00}$ & 1 \\
\texttt{     HS80 } &     136 &      864 &      1662 &  $5.394985\text{E}{-02}$ &  $2.6\text{E}{-09}$ &  $3.7\text{E}{-05}$ &  $1.0\text{E}{+00}$ & 1 \\
\texttt{     HS81 } &      69 &      497 &       294 &  $5.394986\text{E}{-02}$ &  $2.0\text{E}{-05}$ &  $9.5\text{E}{-05}$ &  $8.1\text{E}{-01}$ & 1 \\
\texttt{     HS83 } &    1025 &    99668 &     34494 & -$2.539096\text{E}{+04}$ &  $2.5\text{E}{+00}$ &  $1.8\text{E}{+45}$ &  $9.1\text{E}{-49}$ & -1 \\
\texttt{     HS84 } &    1025 &   100970 &      1083 & -$2.325944\text{E}{+09}$ &  $3.8\text{E}{+05}$ &  $1.7\text{E}{+45}$ &  $2.0\text{E}{-47}$ & -1 \\
\texttt{     HS85 } &    1025 &     5148 &      3315 &  $4.374488\text{E}{+01}$ &  $9.3\text{E}{+06}$ &  $1.0\text{E}{+00}$ &  $5.7\text{E}{-04}$ & -1 \\
\texttt{     HS86 } &      25 &      146 &        64 & -$3.234868\text{E}{+01}$ &  $2.2\text{E}{-16}$ &  $9.8\text{E}{-05}$ &  $5.6\text{E}{-02}$ & 1 \\
\texttt{     HS87 } &      15 &       52 &        16 &  $8.997184\text{E}{+03}$ &  $1.6\text{E}{-09}$ &  $2.3\text{E}{-07}$ &  $1.0\text{E}{-04}$ & 1 \\
\texttt{     HS88 } &      59 &       98 &       163 &  $1.349683\text{E}{+00}$ &  $1.2\text{E}{-05}$ &  $8.3\text{E}{-05}$ &  $1.0\text{E}{-04}$ & 1 \\
\texttt{     HS89 } &     187 &      545 &       778 &  $1.357072\text{E}{+00}$ &  $5.4\text{E}{-06}$ &  $9.1\text{E}{-05}$ &  $1.0\text{E}{-04}$ & 1 \\
\texttt{      HS9 } &       4 &       10 &         6 & -$5.000000\text{E}{-01}$ &  $0.0\text{E}{+00}$ &  $8.5\text{E}{-09}$ &  $1.0\text{E}{+00}$ & 1 \\
\texttt{     HS90 } &     679 &     3485 &      3923 &  $1.385570\text{E}{+00}$ &  $5.0\text{E}{-06}$ &  $9.9\text{E}{-05}$ &  $1.0\text{E}{-04}$ & 1 \\
\texttt{     HS91 } &     525 &     2535 &      2481 &  $1.357178\text{E}{+00}$ &  $5.2\text{E}{-06}$ &  $7.9\text{E}{-05}$ &  $1.0\text{E}{-04}$ & 1 \\
\texttt{     HS92 } &     421 &     2645 &      1628 &  $1.349981\text{E}{+00}$ &  $1.2\text{E}{-05}$ &  $7.5\text{E}{-05}$ &  $1.0\text{E}{-04}$ & 1 \\
\texttt{     HS93 } &    1025 &     7846 &      6127 &  $1.353296\text{E}{+02}$ &  $1.9\text{E}{-06}$ &  $3.2\text{E}{-03}$ &  $2.2\text{E}{-05}$ & -1 \\
\texttt{     HS95 } &      33 &      206 &        62 &  $1.561953\text{E}{-02}$ &  $0.0\text{E}{+00}$ &  $2.7\text{E}{-17}$ &  $1.0\text{E}{-02}$ & 1 \\
\texttt{     HS96 } &      30 &      206 &        64 &  $1.561953\text{E}{-02}$ &  $0.0\text{E}{+00}$ &  $1.0\text{E}{-16}$ &  $1.0\text{E}{-02}$ & 1 \\
\texttt{     HS97 } &      53 &     1007 &        81 &  $4.071246\text{E}{+00}$ &  $0.0\text{E}{+00}$ &  $7.2\text{E}{-15}$ &  $7.0\text{E}{-04}$ & 1 \\
\texttt{     HS98 } &      36 &      303 &       102 &  $3.135809\text{E}{+00}$ &  $0.0\text{E}{+00}$ &  $1.8\text{E}{-15}$ &  $1.1\text{E}{-03}$ & 1 \\
\texttt{     HS99 } &      57 &      348 &       462 & -$8.310799\text{E}{+08}$ &  $1.0\text{E}{-11}$ &  $6.2\text{E}{-05}$ &  $3.8\text{E}{-01}$ & 1 \\
\texttt{  HS99EXP } &    1025 &    32423 &      1025 &  $0.000000\text{E}{+00}$ &  $5.2\text{E}{+03}$ &  $1.2\text{E}{+00}$ &  $1.0\text{E}{+00}$ & -1 \\
  \end{longtable}

\end{scriptsize}

	
	%
	%
	%

\begin{scriptsize}
\setlength{\tabcolsep}{2pt}
{\tiny
	\begin{longtable}{l|r|r|r|r|r|r|r|r|r|r|r}
		\caption{ Comparison with exact algorithm}
		\label{table.comp}\\
		
		\hline
		\multicolumn{1}{l}{Problem}&   \multicolumn{1}{l}{\# cons}    & \multicolumn{1}{c}{\# vars}         & \multicolumn{1}{c}{\# iter}        & \multicolumn{1}{c}{\# pivot}  & \multicolumn{1}{c}{$\downarrow$ \% }      &  \multicolumn{1}{c}{  \# $f$}  & 
		\multicolumn{1}{c}{   $f(x^*)$} &  \multicolumn{1}{c}{$v(x^*)$}  &  \multicolumn{1}{c}{ KKT  } &  \multicolumn{1}{c}{ $\rho_*$} &
		\multicolumn{1}{c}{Exit}
		\\  \hline
		\hline
		\endfirsthead
		\endhead
		
		\hline\hline
		\endlastfoot
		ARGTRIG &  200&200&     3 &      592 &  1.1\%&       4 &   0.0000E$+$00 &  2.6E$-$06 &  0.0E$+$00 &  1.0E$+$00 &1\\
		ARGTRIG\_ex & 200&200&     3 &      599 &    &     4 &      0.0000E$+$00 &  9.3E$-$08 &  0.0E$+$00 &  1.0E$+$00 &1\\
		\hline
		GMNCASE1 &   175&300&   38 &    13049 &    30.8\%&   142 &     2.6779E$-$01 &  2.0E$-$06 &  9.7E$-$05 &  2.1E$-$01&1\\
		GMNCASE1\_ex &   175&300&   41 &    18857 &     &  190 &  2.6779E$-$01 &  4.0E$-$15 &  6.8E$-$05 &  1.1E$-$01&1\\
		\hline
		GMNCASE2 &   175&1050&  172 &    80617 &   51.0\%&   1229 &    -9.9445E$-$01 &  4.2E$-$15 &  1.0E$-$04 &  3.0E$-$03&1\\
		GMNCASE2\_ex &  175&1050& 335 &   164478 &     & 3179 &  -9.9443E$-$01 &  1.6E$-$06 &  9.9E$-$05 &  1.5E$-$01&1\\
		\hline
		DTOC6 &   100&201&   94 &    23430 &   57.1\%&    258 &    7.2798E$+$02 &  1.2E$-$05 &  4.8E$-$04 &  4.9E$-$03 &1\\
		DTOC6\_ex &  100&201&   190 &    54667 &    &    902 &    7.2798E$+$02 &  6.9E$-$06 &  4.7E$-$04 &  3.9E$-$03&1 \\
		\hline
		EIGMAXA &   101&101&    3 &      305 &   $-$1.0\%&  4 & $-$1.0000E$+$00 &  0.0E$+$00 &  0.0E$+$00 &  1.0E$+$00 &1\\
		EIGMAXA\_ex &  101&101&    3 &      302 &     &  38 &  $-$1.0000E$+$00 &  2.8E$-$17 &  4.5E$-$28 &  1.0E$+$00&1 \\
		\hline
		EIGMINA &   101&101&    6 &      986 &   $-$27.7\%&     41 & 1.0000E$+$00 & 1.9E$-$17 &  9.2E$-$29 &  6.2E$-$01 &1\\
		EIGMINA\_ex &  101&101&    7 &      713 &     &     8 &  1.0000E$+$00 &  4.5E$-$14 &  1.7E$-$07 &  5.7E$-$01&1 \\
		\hline
		LUKVLE3 &   2&100&   48 &      879 &  46.7\%&     100 &  2.7587E$+$01 &  1.8E$-$05 &  8.6E$-$05 &  4.2E$-$02 &1\\
		LUKVLE3\_ex &  2&100&    54 &     1650 &   &      91 &     2.7586E$+$01 &  5.7E$-$05 &  5.7E$-$05 &  3.8E$-$02&1 \\
		\hline
		LUKVLE5 &    96&102&  95 &    26275 &   59.8\%&   188 &   2.6393E$+$00 &  7.2E$-$05 &  1.0E$-$04 &  1.0E$+$00 &1\\
		LUKVLE5\_ex & 96&102&    143 &    65431 &    &    350 &   2.6393E$+$00 &  5.6E$-$05 &  8.8E$-$05 &  5.9E$-$01 &1\\
		\hline
		LUKVLE6 &   49&99&   20 &     1442 &     50.1\%&   25 &   6.0399E$+$03 &  9.1E$-$04 &  4.0E$-$05 &  1.0E$+$00 &1\\
		LUKVLE6\_ex & 49&99&     20 &     2887 &    &     30 &   6.0376E$+$03 &  7.4E$-$04 &  4.2E$-$06 &  6.6E$-$01 &1\\
		\hline
		LUKVLE7 &  4&100&   319 &    11080 &    78.3\%&   1777 &    $-$2.5945E$+$01 &  1.7E$-$05 &  1.0E$-$04 &  5.2E$-$02 &1\\
		LUKVLE7\_ex & 4&100&    513 &    51103 &    &   3252 &    $-$2.5944E$+$01 &  1.3E$-$05 &  2.9E$-$04 &  2.8E$-$02&$-$1 \\
		\hline
		LUKVLE8 &   98&100&   157 &    15970 &    $-$15.6\%&   244 &  1.0587E$+$03 &  1.7E$-$07 &  6.5E$-$05 &  2.7E$-$04 &1\\
		LUKVLE8\_ex & 98&100&     128 &    18920 &    &    163 &  1.0587E$+$03 &  1.5E$-$07 &  3.1E$-$05 &  2.5E$-$04 &1\\
		\hline
		LUKVLI6 &   49&99&   23 &     1680 &     42.6\%&   29 &   6.0386E$+$03 &  2.8E$-$04 &  2.7E$-$05 &  1.0E$+$00 &1\\
		LUKVLI6\_ex & 49&99&     20 &     2927 &    &     30 &   6.0376E$+$03 &  7.4E$-$04 &  4.2E$-$06 &  6.6E$-$01 &1\\
		\hline
		LUKVLI8 &  98&100&   145 &    16372 &   55.7\%&    259 &  1.0587E$+$04 &  1.1E$-$07 &  7.2E$-$05 &  2.9E$-$04 &1\\
		LUKVLI8\_ex & 98&100&    255 &    36952 &   &     715 &  1.0587E$+$04 &  4.2E$-$08 &  7.1E$-$05 &  2.8E$-$04 &1\\
		\hline
		NONSCOMPL &  0&500&    24 &     1061 &   27.8\%&     53 &  2.7000E$-$05 &  0.0E$+$00 &  5.6E$-$05 &  1.0E$+$00 &1\\
		NONSCOMPL\_ex & 0&500&    26 &     1470 &    &     70 &   2.3100E$-$05 &  0.0E$+$00 &  7.4E$-$05 &  1.0E$+$00 &1\\
		\hline
		NONSCOMPNE &  100&100&     7 &      695 &  $-$9.1\%&       8 &  0.0000E$+$00 &  8.1E$-$07 &  1.8E$-$11 &  1.0E$+$00 &1\\
		NONSCOMPNE\_ex & 100&100&     6 &      643 &   &       7 & 0.0000E$+$00 &  4.2E$-$08 &  1.2E$-$06 &  1.0E$+$00 &1\\
		\hline
		ORTHREGC &  50&105&    27 &     3712 & 92.1\%&       53 &   1.9756E$+$00 &  9.9E$-$04 &  9.8E$-$04 &  1.0E$+$00 &1\\
		ORTHREGC\_ex & 50&105&   205 &    46737 &  &     1231 &   1.9755E$+$00 &  8.4E$-$05 &  7.4E$-$04 &  1.0E$+$00 &1\\
		\hline
		ORTHREGD &     50&103&   82 &    12336 &   85.6\%&    173 &     1.5590E$+$01 &  9.1E$-$04 &  4.2E$-$06 &  1.0E$+$00 &1\\
		ORTHREGD\_ex & 50&103&    513 &    85794 &  &     3623 &     3.9532E$+$01 &  8.7E$-$04 &  3.3E$-$01 &  1.5E$-$02 &$-$1\\
		\hline
		ORTHREGF &     49&152&   39 &     5406 &  60.8\%&      66 &   1.3152E$+$00  &  3.9E$-$04 &  9.0E$-$04 &  1.0E$+$00 &1\\
		ORTHREGF\_ex & 49&152&     48 &    13792 &  &      155 &      1.3150E$+$00  &  4.4E$-$04 &  9.7E$-$04 &  1.0E$+$00 &1\\
		\hline
		OSCIGRNE &   0&1000&    6 &      752 &   0.0\%&      7 &  0.0000E$+$00 &  2.4E$-$08 &  8.3E$-$25 &  1.0E$+$00&1 \\
		OSCIGRNE\_ex &  0&1000&    6 &      752 &    &      7 &  0.0000E$+$00 &  2.4E$-$08 &  8.3E$-$25 &  1.0E$+$00 &1\\
		\hline
		PENLT1NE &  101&100&    40 &      980 &  52.3\%&      62 &0.0000E$+$00 &  9.9E$-$04 &  1.6E$-$08 &  1.0E$+$00 &1\\
		PENLT1NE\_ex & 101&100&    36 &     2054 &   &      44 &  0.0000E$+$00 &  9.9E$-$04 &  9.6E$-$09 &  1.0E$+$00 &1\\
		\hline
		PRIMAL1 &  85&325&    42 &    21042 &  0.0\%&     232 &   $-$3.5010E$-$02 &  4.5E$-$14 &  1.4E$-$04 &  1.0E$+$00 &1\\
		PRIMAL1\_ex & 85&325&    42 &    21042 &   &     232 &     $-$3.5010E$-$02 &  4.5E$-$14 &  1.4E$-$04 &  1.0E$+$00 &1\\
		\hline
		SOSQP1 &  101&200&     3 &      373 &  7.2\%&        4 & 0.0000E$+$00 &  0.0E$+$00 &  0.0E$+$00 &  1.0E$+$00 &1\\
		SOSQP1\_ex & 101&200&     3 &      402 &   &       4 &  0.0000E$+$00 &  0.0E$+$00 &  0.0E$+$00 &  1.0E$+$00 &1\\
		\hline
		STCQP2 &   101&200&   74 &    13022 &    11.0\%&   317 &  1.4294E$+$03 &  0.0E$+$00 &  4.7E$-$04 &  2.3E$-$07 &1\\
		STCQP2\_ex &  101&200&   77 &    14624 &    &   111 &  1.4294E$+$03 &  0.0E$+$00 &  4.0E$-$04 &  1.1E$-$03 &1\\
	\end{longtable}
}
\end{scriptsize}

	
	

	\section{Appendix}
	
	\noindent\textbf{Calculation  of the dual subproblem} 
	
	In this appendix, we show how the dual subproblem \eqref{dual prob} is derived.  
	For this purpose, we use the shorthand defined by \eqref{eq.shorthand} and  define the convex sets: 
	\[ \begin{aligned}
	C_i = &\ \{ d \mid  \langle a_i /\|a_i\|_2 , d \rangle + b_i/\|a_i\|_2  = 0\}, \ i\in \Ecal\\
	C_i =& \ \{ d \mid  \langle a_i/\|a_i\|_2 , d \rangle + b_i /\|a_i\|_2 \le 0\}, \  i\in\Ical. 
	\end{aligned}\]
	One finds that 
	$\delta^*(u_i \mid C_i) < +\infty$ if and only if 
	\[  \langle u_i, a_i  \rangle = \pm \| u_i \|_2 \| a_i \|_2, \  i\in\Ecal, \quad \text{and}\quad  \langle u_i, a_i  \rangle =   \| u_i \|_2 \| a_i \|_2, \  i\in\Ical.\]
	In this case, it must be true that 
	\begin{equation}\label{comp dual}
	\begin{aligned}
	u_i & = \lambda_i a_i /\|a_i\|_2  \ \text{for some }  \lambda_i \in \mathbb{R}, \quad \text{ so that } \delta(u_i \mid C_i) = - \lambda_i b/\|a_i\|_2 , \quad \text{for }\  i\in \Ecal,\\
	u_i & = \lambda_i a_i /\|a_i\|_2  \ \text{for some }  \lambda_i \in \mathbb{R}_+, \  \text{ so that } \delta(u_i \mid C_i) = - \lambda_i b/\|a_i\|_2 , \quad \text{for }\  i\in \Ical.\\
	\end{aligned}
	\end{equation}
	For simplicity, suppose $\Ecal=\{1,\ldots, \bar m\}$ and $\Ical=\{\bar m+1,\ldots, m\}$.  The primal problem thus can be written as  
	\[
	\inf_d\ \langle g, d\rangle +   \delta(d\mid X)+\sum_{i=1}^m \|a_i\|_2 \text{dist}(d\mid C_i).
	\]
	The primal problem thus can be rewritten as
	\[
	\inf_{d, \zbf}\ \phi(d, \zbf) := \langle g, d\rangle+ \delta(d\mid X)+
	\sum_{i=1}^m \|a_i\|_2 [\support{d-z_i}{\mathbb{B}_2}+  \delta(z_i\mid C_i)].
	\]
	with $\zbf = [z_1^T, \ldots, z_m^T]^T$
	by noticing that the optimal $z_i\in\mathbb{R}^n$ is always the projection of $d$ onto $C_i$.

	To derive the dual, notice that the primal objective is equivalent to 
	\[\phi(d, \zbf) =   \sup_{\ubf, \vbf} \Lcal((d, \zbf), (\ubf, \vbf))\]
	with $u_i\in\mathbb{R}^n, i=1,\ldots, m$, $v_i\in\mathbb{R}^n, i=1,\ldots, m, m+1$, $\ubf = [u_1^T, \ldots, u_m^T]^T$,  $\vbf = [v_1^T, \ldots, v_m^T, v_{m+1}^T]^T$ and the associated Lagrangian is given by
	\[
	\begin{aligned}
	\Lcal((d, \zbf), (\ubf, \vbf)) 
	= &\  \langle g, d\rangle +\langle v_{m+1}, d\rangle -\support{v_{m+1}}{X}\\
	& \ +\sum_{i=1}^m\|a_i\|_2[\langle u_i, d-z_i\rangle - \delta(u_i\mid \mathbb{B}_2)+\langle v_i, z_i\rangle-\support{v_i}{C_i}]\\
	= & \ \langle g + v_{m+1} + \sum_{i=1}^m \|a_i\|_2 u_i, d\rangle  -\support{v_{m+1}}{X} \\ 
	& \ +\sum_{i=1}^m\|a_i\|_2[ \langle - u_i + v_i, z_i\rangle -\delta(u_i\mid \mathbb{B}_2) - \support{v_i}{C_i}].
	\end{aligned}
	\]
	Hence, the dual objective is then
	\[
	\psi(\ubf, \vbf):=\inf_{d,\zbf}\Lcal((d,\zbf),(\ubf,\vbf)),
	\]
	The optimality conditions for the infimum in the definition of $\psi$ are
	\[
	\begin{aligned}
	\nabla_d \Lcal & =g+ v_{m+1}+\sum_{i=1}^m\|a_i\|_2 u_i= 0
	\\
	\nabla_{z_i}\Lcal &=- \|a_i\|_2  u_i+ \|a_i\|_2 v_i=0\ \qquad i=1,\dots m.
	\end{aligned}
	\]
	Eliminating the $d_i, z_i$ and $v_i$ variables in $\Lcal$ using the above conditions  yields the dual objective
	\[
	\psi(\ubf, \vbf)= -\support{-g-\sum_{i=1}^m \|a_i\|_2 u_i}{X}-
	\sum_{i=1}^m \|a_i\|_2 [\delta(u_i\mid \mathbb{B}_2)+\support{u_i}{C_i}],
	\]
	which yields  the associated  dual problem
	\[
	\begin{aligned}
	&\inf_{\ubf}\   \support{-g-\sum_{i=1}^m \|a_i\|_2 u_i}{X}+\sum_{i=1}^m \|a_i\|_2 \support{u_i}{C_i}
	\\
	&\mbox{s.t.}\ u_i\in \mathbb{B}_2\,\quad i=1,\dots,m.
	\end{aligned}
	\]

	We now use \eqref{comp dual} to further simplify the dual problem, so that the dual problem is given by 
	\[
	\begin{aligned}
	\inf_\lambda&\   \support{-g-\sum_{i=1}^m \lambda_i a_i}{X} - \sum_{i=1}^m b_i\lambda_i 
	\\
	\mbox{s.t.}&\ -1\le \lambda_i \le 1,\ i\in\Ecal,\\
	&\quad\  0 \le \lambda_i \le 1,\ i\in\Ical.
	\end{aligned}
	\]
	
	For the support function of norm ball $X = \{ x \mid  \|x \| \le \gamma\}$ with radius $\gamma > 0$, we have 
	$\delta^*(v \mid X) = \gamma \| v \|_*$. 
	Hence, we finally derive our dual problem 
	\[
	\begin{aligned}
	\inf_\lambda&\  \gamma \| g-\sum_{i=1}^m \lambda_i a_i \|_* - \sum_{i=1}^m b_i\lambda_i 
	\\
	\mbox{s.t.}&\ \lambda_i\in [-1,1],\ i\in\Ecal,\\
	&  \ \lambda_i\in [0,1],\quad i\in\Ical.
	\end{aligned}
	\]
	
	In addition, general convex analysis tells that 
	\begin{itemize}
		\item The dual of $\ell_2$ norm is the $\ell_2$ norm.
		\item The dual of $\ell_1$ norm is the $\ell_\infty$ norm.
		\item The dual of $\ell_\infty$ norm is the $\ell_1$ norm.
		\item More generally, the dual of the $\ell_p$ norm is the $\ell_q$ norm with $1/p + 1/q  = 1$. 
		\item For matrix norms, the dual of $\ell_2$ (or, the spectral) norm is the nuclear norm. 
		\item The dual of the Frobenius norm  is the Frobenius norm.
	\end{itemize}
	Hence, if we use an $\ell_2$ norm ball for the trust region, then the dual problem is quadratic; if we use an $\ell_\infty$ norm ball for the trust 
	region, then the dual problem has an $\ell_\infty$ norm in the objective.

	\bibliographystyle{tfs}
	\bibliography{reference}

\end{document}